\documentclass[10pt]{amsart}
\numberwithin{equation}{section}

\usepackage{amsmath}
\usepackage{amscd,amsthm,amssymb,amsfonts}
\usepackage{mathrsfs}
\usepackage{dsfont}
\usepackage{stmaryrd}
\usepackage{euscript}
\usepackage{expdlist}
\usepackage{enumerate}

\input xy

\newtheorem{thm}{Theorem}[section]
\newtheorem{thmA}{Theorem}

\newtheorem{assumption}[thm]{Assumption}

\newtheorem*{thm*}{Theorem}

\newtheorem{lm}[thm]{Lemma}
\newtheorem{cor}[thm]{Corollary}
\newtheorem*{cor*}{Corollary}
\newtheorem{prop}[thm]{Proposition}
\newtheorem*{conj*}{Conjecture}

\theoremstyle{Remark}

\newtheorem*{remark}{Remark}

\theoremstyle{definition}
\newtheorem*{defn*}{Definition}

\newtheorem{I_Remark*}{Remark}

\newcommand{\nc}{\newcommand}

\newcommand{\beq}{\begin{equation}}
\newcommand{\eeq}{\end{equation}}
\newcommand{\bpmx}{\begin{pmatrix}}
\newcommand{\epmx}{\end{pmatrix}}
\newcommand{\bbmx}{\begin{bmatrix}}
\newcommand{\ebmx}{\end{bmatrix}}

\def\parref#1{\ref{#1}}
\def\thmref#1{Theorem~\parref{#1}}

\def\propref#1{Proposition~\parref{#1}}
\def\corref#1{Corollary~\parref{#1}}

\def\lmref#1{Lemma~\parref{#1}}

\def\makeop#1{\expandafter\def\csname#1\endcsname
  {\mathop{\rm #1}\nolimits}\ignorespaces}

\makeop{Hom}   \makeop{End}   \makeop{Aut}   
\makeop{Pic} \makeop{Gal}       \makeop{Div} \makeop{Lie}
\makeop{PGL}   \makeop{Corr} \makeop{PSL} \makeop{sgn} \makeop{Spf}
 \makeop{Tr} \makeop{Nr} \makeop{Fr} \makeop{disc}
\makeop{Proj} \makeop{supp} \makeop{ker}   \makeop{Im} \makeop{dom}
\makeop{coker} \makeop{Stab} \makeop{SO} \makeop{SL} \makeop{SL}
\makeop{Cl}    \makeop{cond} \makeop{Br} \makeop{inv} \makeop{rank}
\makeop{id}    \makeop{Fil} \makeop{Frac}  \makeop{GL} \makeop{SU}
\makeop{Trd}   \makeop{Sp} \makeop{Tr}    \makeop{Trd} \makeop{Res}
\makeop{ind} \makeop{depth} \makeop{Tr} \makeop{st} \makeop{Ad}
\makeop{Int} \makeop{tr}    \makeop{Sym} \makeop{can} \makeop{SO}
\makeop{torsion} \makeop{GSp} \makeop{Tor}\makeop{Ker} \makeop{rec}
\makeop{Ind} \makeop{Coker}
 \makeop{vol} \makeop{Ext} \makeop{gr} \makeop{ad}
 \makeop{Gr}\makeop{corank} \makeop{Ann}
\makeop{Hol} 
\makeop{Fitt} \makeop{Mp} \makeop{CAP}

\def\makebb#1{\expandafter\def
  \csname bb#1\endcsname{{\mathbb{#1}}}\ignorespaces}
\def\makebf#1{\expandafter\def\csname bf#1\endcsname{{\bf
      #1}}\ignorespaces}
\def\makegr#1{\expandafter\def
  \csname gr#1\endcsname{{\mathfrak{#1}}}\ignorespaces}
\def\makescr#1{\expandafter\def
  \csname scr#1\endcsname{{\EuScript{#1}}}\ignorespaces}
\def\makecal#1{\expandafter\def\csname cal#1\endcsname{{\mathcal
      #1}}\ignorespaces}

\def\doLetters#1{#1A #1B #1C #1D #1E #1F #1G #1H #1I #1J #1K #1L #1M
                 #1N #1O #1P #1Q #1R #1S #1T #1U #1V #1W #1X #1Y #1Z}
\def\doletters#1{#1a #1b #1c #1d #1e #1f #1g #1h #1i #1j #1k #1l #1m
                 #1n #1o #1p #1q #1r #1s #1t #1u #1v #1w #1x #1y #1z}
\doLetters\makebb   \doLetters\makecal  \doLetters\makebf
\doLetters\makescr
\doletters\makebf   \doLetters\makegr   \doletters\makegr

    \def\setminus{\smallsetminus}

\normalsize

\makeop{Ram} \makeop{Rep} \makeop{mass}

\makeop{Bl}

\def\diag#1{\mathrm{diag}(#1)}

\def\cB{\EuScript B}

\def\cH{{\mathcal H}}
\def\cI{\mathcal I}

\def\cR{{\mathcal R}}

\def\cS{{\mathcal S}}

\def\cV{{\mathcal V}}

\def\cY{\mathcal Y}

\def\sH{\mathscr H}

\def\sB{\mathscr B}

\def\sS{\mathscr S}
\def\sP{\mathscr P}
\def\sW{{\mathscr W}}

\def\ot{\otimes}

\def\hookto{\hookrightarrow}
\def\longto{\longrightarrow}
\def\ol{\overline}  \nc{\opp}{\mathrm{opp}} \nc{\ul}{\underline}

\def\XYmatrix{\xymatrix@M=8pt} 
\def\ncmd{\newcommand}
\ncmd{\xysubset}[1][r]{\ar@<-2.5pt>@{^(-}[#1]\ar@<2.5pt>@{_(-}[#1]}
\ncmd{\XYmatrixc}[1]{\vcenter{\XYmatrix{#1}}}
\ncmd{\xyto}[1][r]{\ar@{->}[#1]}
\ncmd{\xyinj}[1][r]{\ar@{^(->}[#1]}
\ncmd{\xysurj}[1][r]{\ar@{->>}[#1]}
\ncmd{\xyline}[1][r]{\ar@{-}[#1]}
\ncmd{\xydotsto}[1][r]{\ar@{.>}[#1]}
\ncmd{\xydots}[1][r]{\ar@{.}[#1]}
\ncmd{\xyleadsto}[1][r]{\ar@{~>}[#1]}
\ncmd{\xyeq}[1][r]{\ar@{=}[#1]} \ncmd{\xyequal}[1][r]{\ar@{=}[#1]}
\ncmd{\xyequals}[1][r]{\ar@{=}[#1]}
\ncmd{\xymapsto}[1][r]{l\ar@{|->}[#1]}\ncmd{\xyimplies}[1][r]{\ar@{=>}[#1]}
\ncmd{\xyiso}{\ar[r]_-{\sim}}
\def\injxy{\ar@{^(->}}

\newcommand{\pMX}[4]{\begin{pmatrix}
{#1}& {#2}\\
{#3}&{#4}\end{pmatrix} }

\newcommand{\seesaw}[4]{{#1}\ar@{-}[rd]\ar@{-}[d]&{#2}\ar@{-}[d]\\
{#3}\ar@{-}[ru]&{#4}}

\def\x{{\times}}

\def\e{\varepsilon} 

\newcommand\stt[1]{\left\{#1\right\}}

\renewcommand\pmod[1]{\,(\mbox{mod }{#1})}

\renewcommand\Re{\text{Re}\,}

\usepackage{color}
\usepackage{hyperref}
\usepackage{MnSymbol}

\def\U{{\rm U}}
\def\e{\epsilon}
\def\la{\lambda}
\def\itPi{\mathit{\Pi}}

\def\b{\bar}
\def\ul{\underline}
\def\t{\tilde}
\def\a{\alpha}

\title[Local newforms of unramified $\U_{2n+1}$ and Rankin-Selberg integrals]
{Local newforms for generic representations of unramified ${\rm U}_{2n+1}$ and Rankin-Selberg integrals}
\author{Yao Cheng}
\date{\today}
\address{No. 151, Yingzhuan Road, Tamsui District, New Taipei City 251, Taiwan (R.O.C),  Lui-Hsien Memorial 
Science Hall.}
\email{briancheng@mail.tku.edu.tw}
\begin{document}
\maketitle
\begin{abstract}
Recently Atobe-Oi-Yasuda established the newform theory for irreducible tempered generic representations of 
unramified $\U_{2n+1}$ over non-archimedean local fields. In this paper we extend their result to every irreducible generic
representations and compute the dimensions of the spaces of oldforms. We also compute the Rankin-Selberg integrals 
attached to newforms and oldforms under a natural assumption on the $\gamma$-factors defined by the Rankin-Selberg 
integrals.
\end{abstract}

\section{Introduction}
Newforms have their roots in the classical theory of Atkin-Lehner for modular forms. In that setting, 
newforms are cusp forms which are simultaneously eigenfunctions of all Hecke operators. As a consequence, their Fourier 
coefficients satisfy strong recurrence relations, and their $L$-functions are well behaved. On the other hand, 
oldforms are cusp forms originated from newforms of lower level via certain level raising procedures. Probably the first 
definition and attempt to study newforms in the local $p$-adic setting dates back to Casselman's work
(\cite{Casselman1973}) on $\GL_2$. His result singles out, in every irreducible $generic$ (complex) representation of 
$\GL_2$, a unique (up to scalars) non-zero vector. These vectors are called local newforms as they appeared as local 
components of adelizations of modular newforms. Consequently, local newforms can be used to study automorphic 
forms on $\GL_2$ as well as their applications. They also play indispensable roles for the ramification theory of 
representations of $p$-adic $\GL_2$.\\

Casselman's result was subsequently extended to irreducible $generic$ representations of $p$-adic $\GL_r$
by Jacquet--Piatetski-Shapiro--Shalika (\cite{JPSS1981}, see also \cite{Jacquet2012}, \cite{Matringe2013}).
Their method was based on the Rankin-Selberg integrals constructed in \cite{JPSS1983}.
Then built on the result of Jacquet et al, Reeder (\cite{Reeder1991}) studied oldforms for generic representations of 
$\GL_r$; he showed that oldforms can be produced from a newform via level raising operators and obtained the 
dimensions of the spaces of oldforms. Recently, Humphries gave alternative characterizations of local newofrms of generic
representations of $\GL_r$ in terms of $K$-types (\cite{Humphries2022}); Atobe-Kondo-Yasuda established the theory 
of local newforms for $every$ irreducible representation of $p$-adic $\GL_r$ in \cite{AtobeKondoYasuda}.\\ 

For other $p$-adic classical groups, Roberts-Schmidt developed the theory of local newforms and oldforms for what they
called $paramodular$ representations of ${\rm PGSp}_4$ (\cite{RobertsSchmidt2007}). These include all irreducible 
generic representations and some of non-generic ones. 
For generic representations, they in addition computed the zeta (cf. \cite{Novodvorsky1979}) integrals attached to 
newforms and oldforms (see also \cite{YCheng2022}).  By adopting the arguments in \cite{RobertsSchmidt2007}, 
Miyauchi obtained the theory of local newforms and oldforms for generic representations of unramified ${\rm U}_{2,1}$ 
in the series of papers (\cite{Miyauchi2013a}, \cite{Miyauchi2013b}, \cite{Miyauchi2013c}, \cite{Miyauchi2018}), in which 
he also computed the zeta integrals (cf. \cite{GPS1983}) attached to newforms. 

\subsection{Results of Atobe-Oi-Yasuda}
In a recent preprint \cite{AtobeOiYasuda}, Atobe-Oi-Yasuda extended Miyauchi's result and established the theory of local 
newforms for irreducible $tempered$ generic representations of unramfied ${\rm U}_{2n+1}$ over non-archimedean local 
fields. To state their results, let $F$ be a finite field extension of $\mathbb{Q}_p$ with $p>2$ and $E$ be the unramified 
quadratic field extension of $F$. Denote by $x\mapsto\bar{x}$ the action of the non-trivial element in the Galois group 
${\rm Gal}(E/F)$ on $E$. This action and its notation extend naturally to matrices with entries in $E$. Let $\frak{o}_F$ 
(resp. $\frak{o}_E$) be the valuation ring of $F$ (resp. $E$), $\frak{p}_F$ (resp. $\frak{p}_E$) be its maximal ideal and
$q_F=|\frak{o}_F/\frak{p}_F|$ (resp. $q_E=|\frak{o}_E/\frak{p}_E|=q^2$). Fix an element $\delta\in\frak{o}_E^\x$ with 
$\bar{\delta}=-\delta$ and an additive character $\psi_F$ of $F$ that is trivial on $\frak{o}_F$ but not on $\frak{p}_F^{-1}$. 
Put $\psi_E(x)=\psi_F(\frac{x-\b{x}}{2\delta})$ for $x\in E$. Then $\psi_E$ defines an additive character of $E$ which is 
trivial on $F$ and $\frak{o}_E$ but not on $\frak{p}_E^{-1}$.\\

Let $J_N\in\GL_N(E)$ be the element defined inductively by 
\[
J_1=(1)
\quad\text{and}\quad
J_N
=
\begin{pmatrix}
&J_{N-1}\\
1&
\end{pmatrix}.
\]
Then one has the unitary group ${\rm U}_{2n+1}(F)\subset {\rm GL}_{2n+1}(E)$ defined by 
\[
{\rm U}_{2n+1}(F)=\{g\in{\rm GL}_{2n+1}(E)\mid {}^t\bar{g} J_{2n+1} g=J_{2n+1}\}.
\] 
For an integer $m\ge 0$, let $K_{n,m}\subset {\rm U}_{2n+1}(F)$ be the open compact subgroup given by 
\[
K_{n,m} 
=
\bordermatrix{
              & n     & 1    & n     \cr
    n     & \frak{o}_E &\frak{o}_E &\frak{p}_E^{-m}     \cr
    1     &\frak{p}^m_E& 1+\frak{p}^m_E &\frak{o}_E\cr
    n     &\frak{p}^m_E&\frak{p}^m_E& \frak{o}_E   \cr
            }\cap {\rm U}_{2n+1}(F).
\]
These $K_{n,m}$ are the open compact subgroups considered in \cite{AtobeOiYasuda}. Notice that the matrices used to 
defined the unitary groups in this paper and in \cite{AtobeOiYasuda} are different. However, the associated unitary groups 
are conjugate by a diagonal element in $\GL_{2n+1}(\frak{o}_E)$; hence $K_{n,m}$ remains unchanged for all $m$.\\
 
Let $\pi$ be an irreducible complex representation of ${\rm U}_{2n+1}(F)$. Then by results of Mok (\cite{Mok2015}) and 
Gan-Gross-Prasad (\cite[Theorem 8.1]{GanGrossPrasad2012}), one can attached to $\pi$ an $(2n+1)$-dimensional 
$conjugate$-$orthogonal$ complex representation $\phi_\pi$ of the Weil-Deligne group of $E$. Let 
$\epsilon(s,\phi_\pi,\psi_E)$ denotes the $\epsilon$-factor attached to $\phi_\pi$ and $\psi_E$ (cf. \cite{Tate1979}). 
Then since $\psi_E$ is unramified and trivial on $F$, it can be written as 
\begin{equation}\label{E:epsilon}
\epsilon(s,\phi_\pi,\psi_E)
=
q_E^{-a_\pi(s-1/2)}
\end{equation}
for some integer $a_\pi\ge 0$ (cf. \cite[Proposition 5.2 (2)]{GanGrossPrasad2012}). 
Now their results can be stated as follow.

\begin{thm}[Atobe-Oi-Yasuda]\label{T:main'}
Let $(\pi,\cV_\pi)$ be an irreducible tempered representation of ${\rm U}_{2n+1}(F)$. 
\begin{itemize}
\item[(1)]
If $\pi$ is non-generic, then $\cV_\pi^{K_{n,m}}=0$ for all $m$.
\item[(2)]
If $\pi$ is generic, then
\[
{\rm dim}_{\mathbb{C}}\cV_\pi^{K_{n,m}}
=
\begin{cases}
0\quad\text{if $m<a_\pi$},\\
1\quad\text{if $m=a_\pi$}.
\end{cases}
\]
Moreover, if $m>a_\pi$, then $\cV_\pi^{K_{n,m}}\neq 0$.
\end{itemize}
\end{thm}

Let us briefly point out the key ingredients in their ingenious proof. To prove \thmref{T:main'} (1), they applied the local 
Gan-Gross-Prasad conjecture (\cite{GanGrossPrasad2012}) established by Beuzart-Plessis
(\cite{Beuzart-Plessis2014}, \cite{Beuzart-Plessis2015}, \cite{Beuzart-Plessis2016}) and the following property of 
$K_{n,m}$. If we put
\[
R_{n,m}=\U_{2n}(F)\cap K_{n,m}
\]
then $R_{n,m}$ is a hyperspecial maximal compact subgroup of $\U_{2n}(F)$, up to conjugate an element in the diagonal 
torus of $\U_{2n}(F)$. Here we identify the unramified $\U_{2n}(F)$ with a subgroup of $\U_{2n+1}(F)$ fixing an 
anisotropic vector. To prove \thmref{T:main'} (2), they applied the endoscopy character relation established by Mok 
(\cite{Mok2015}) as well as the theory of local newforms established by Jacquet et al (\cite{JPSS1981}). 
For this, they proved an analogue of the fundamental lemma for the open compact subgroups $K_{n,m}$.

\subsection{Main results of this paper}
In this paper we further develop the theory of local newforms of unramified $\U_{2n+1}$ initiated by Atobe-Oi-Yausda. 
Our goal is twofold. First, we extend their results from tempered generic representations to every generic representations 
and compute the dimensions of the spaces of oldforms. Second, we compute the Rankin-Selberg integrals 
(cf. \cite{Ben-ArtziSoudry2009}) attached to newforms and oldforms under a natural assumption on the $\gamma$-factors
defined by the Rankin-Selberg integrals. 

\subsubsection{Dimension formula}
Our first result is the following dimension formula for every generic representations.

\begin{thm}\label{T:main}
Let $(\pi,\cV_\pi)$ be an irreducible generic representation of ${\rm U}_{2n+1}(F)$. Then we have 
\footnote{Here we understand that $\begin{pmatrix}a\\b\end{pmatrix}=0$ when $a<b$.}
\begin{equation}\label{E:dim formula}
{\rm dim}_{\mathbb{C}}\cV_\pi^{K_{n,m}}
=
\begin{pmatrix}
\lfloor\frac{m-a_\pi}{2}\rfloor+n\\
n
\end{pmatrix}
\end{equation}
for every integer $m\ge 0$.
\end{thm}

Let us make a few remarks:

\begin{itemize}
\item[(1)]
Following the usual convention, non-zero elements in the line $\cV_\pi^{K_{n,a_\pi}}$ are called $newforms$ of $\pi$;
elements in $\cV_\pi^{K_{n,m}}$ for $m>a_\pi$ are called $oldforms$ of $\pi$.
\item[(2)]
The proof of \thmref{T:main} consists of two steps. The first step is to obtain \eqref{E:dim formula} when $\pi$ is tempered,
which follow from the results of Reeder (\cite[Theorem 1]{Reeder1991}) and Atobe-Oi-Yasada 
(\cite[Theorem 1.2]{AtobeOiYasuda}). The next step is to extend \eqref{E:dim formula} to non-tempered $\pi$;
such $\pi$ is isomorphic to a full induced representation of $\U_{2n+1}(F)$ by the standard 
module conjecture (cf. \cite{CasselmanShahidi1998}, \cite{Muic2001}). Therefore, our task is to compute the double coset
decompositions of $\U_{2n+1}(F)$ by maximal parabolic subgroups and $K_{n,m}$. Similar computations already 
appeared in the works of Roberts-Schmidt (\cite[Chapter 5]{RobertsSchmidt2007}) and Miyauchi (\cite{Miyauchi2013c});
our approach mimics theirs.
\item[(3)]
Equation \eqref{E:dim formula} is compatible with results of Miyauchi in \cite{Miyauchi2013a} and
\cite{Miyauchi2013c}. Consequently, by combining \thmref{T:main} with the results of Miyauchi in 
\cite[Theorem 4.3]{Miyauchi2013b} and \cite{Miyauchi2018}, one sees immediately that the $\epsilon$-factors defined 
by the Rankin-Selberg integrals and the associated Weil-Deligne representation agree.
\item[(4)]
Finally, we point out that similar reduction procedures can be used to reduce the newform 
conjecture for $p$-adic $\SO_{2n+1}$ proposed by Gross (\cite{Gross2015}) and Tsai (\cite{Tsai2013}, \cite{Tsai2016})
from generic representations to tempered generic representations.
\end{itemize}

\subsubsection{Rankin-Selberg integrals}
As mentioned in \cite{AtobeOiYasuda}, their proof does not involve Rankin-Selberg integrals; hence the relation 
between Rankin-Selberg integrals and newforms remains unclear. On the other hand, from results in the literature, 
newforms usually serve as "test vectors" for certain Rankin-Selberg integrals, in the sense that if one computes the 
Rankin-Selberg integral attached to a newform, then one obtains the $L$-factor attached to the generic representation
(cf. \cite{JPSS1981}, \cite{RobertsSchmidt2007}, \cite{Miyauchi2018}, \cite{YCheng2022}). In view of these, another goal 
of this paper is to investigate this relation (under a natural assumption).
The Rankin-Selberg integrals considered here are those attached to generic representations of 
$
\U_{2n+1}\x{\rm Res}_{E/F}\GL_r
$ 
where ${\rm Res}_{E/F}$ stands for the Weil 
restriction from $E$ to $F$. These local integrals are indeed coming from their global counterparts, which were first 
studied by Tamir in \cite{Tamir1991} when $r=n$, following the idea of Gelbart and Piatetski-Shapiro 
(\cite[Part B]{GPSR1987}).  His results were later extended to the case $r<n$ by Ben-Artzi and Soudry in 
\cite{Ben-ArtziSoudry2009}, and to the case $r>n$ by Ginzburg-Rallis-Soudry in \cite{GRS2011}, whose local integrals
were further studied by Morimoto-Soudry in \cite{MorimotoSoudry2020}.\\

To describe our second result, let $(\tau,\cV_\pi)$ be an unramified representation of ${\rm Res}_{E/F}\GL_r(F)=\GL_r(E)$
that is an induced representation of $Langlands$ $type$. This includes all (classes of) irreducible unramified generic 
representations of $\GL_r(E)$, but also contains reducible ones. Moreover, it has the following features: 
(i) the space $\cV_\tau^{\GL_r(\frak{o}_E)}$ is one-dimensional; 
(ii) $\tau$ admits a unique non-zero Whittaker functional $\Lambda_{\tau,\b{\psi}_E}$ which is non-trivial on 
$\cV_\tau^{\GL_r(\frak{o}_E)}$; 
(iii) the unique irreducible quotient $J(\tau)$ of $\tau$ is unramified.
Because of (i) and (ii), we can fix a basis $v_\tau$ of $\cV_\tau^{\GL_r(\frak{o}_E)}$ with
$\Lambda_{\tau,\b{\psi}_E}(v_\tau)=1$. Now let $s$ be a complex number. Then by inflating the representation 
$\tau\ot|\det|_E^{s-1/2}$ to a parabolic subgroup of unramified $\U_{2r}(F)$ of Siegel type, we get a normalized induced 
representation $\rho_{\tau,s}$ of $\U_{2r}(F)$, whose underlying space $I_r(\tau,s)$ consisting of smooth 
$\cV_\tau$-valued functions on $\U_{2r}(F)$ satisfying the usual rule. The Rankin-Selberg integral $\Psi_{n,r}(v\ot\xi_s)$ is 
attached to $v\in\cV_\pi$ and $\xi_s\in I_r(\tau,s)$. By applying the normalized intertwining operator 
on $I_r(\tau,s)$ together with the multiplicity one result of certain Hom space, one gets the $\gamma$-factor 
$\gamma_\delta(s,\pi\x\tau,\psi_F)$ depending on the choice of $\delta$ and $\psi_F$ defined by these local integrals. 
In this paper, we assume the following.

\begin{assumption}\label{H}
Suppose that $1\le r\le n$. Then 
\[
\gamma_\delta(s,\pi\x\tau,\psi_F)=\gamma(s,\phi_\pi\ot\phi_{J(\tau)},\psi_E)
\] 
where $\phi_{J(\tau)}$ is the $L$-parameter of $J(\tau)$ under the local Langlands correspondence of $\GL_r(E)$.
\end{assumption}

Assume $1\le r\le n$ from now on. Then $\U_{2r}(F)$ can be regarded as a subgroup of $\U_{2n+1}(F)$ via a 
natural embedding. For each $m\ge 0$, let us put $R_{r,m}=\U_{2r}(F)\cap K_{n,m}$. 
Then $R_{r,m}$ satisfies the similar property as $R_{n,m}$ mentioned in the previous subsection. In particular, the space 
$I_r(\tau,s)^{R_{r,m}}$ is one-dimensional and admits a generator $\xi_{\tau,s}^m$ with $\xi^m_{\tau,s}(I_{2r})=v_\tau$. 
Now our second result can be stated as follow.

\begin{thm}\label{T:main2}
Let $(\pi,\cV_\pi)$ be an irreducible generic representation of $\U_{2n+1}(F)$ and fix a non-zero Whittaker functional 
$\Lambda_{\pi,\psi_E}$ on $\cV_\pi$ as well as a basis $v_\pi$ of $\cV_\pi^{K_{n,a_\pi}}$. Then under the Assumption
\ref{H}, 
\begin{equation}\label{E:main eqn}
\Psi_{n,r}(v_\pi\ot\xi_{\tau,s}^{a_\pi})
=
\frac{L(s,\phi_\pi\ot\phi_{J(\tau)})}{L(2s, \phi_{J(\tau)}, {\rm As})}
\cdot
\Lambda_{\pi,\psi_E}(v_\pi)
\end{equation}
provided that the Haar measures are suitably chosen (cf. \S\ref{SS:setup}). Here As stands for the Asai representation
\footnote{We slightly abuse the notation to regard $\phi_{J(\tau)}$ as the $L$-parameter of the representation $J(\tau)$ of 
the $F$-group ${\rm Res}_{E/F}\GL_r(F)$.} (cf. \cite[p.20-p.21]{GRS2011}). Moreover, if $\pi$ is 
tempered, then $\Lambda_{\pi,\psi_E}(v_\pi)$ is non-zero.
\end{thm}

Again, let us make a few remarks:
\begin{itemize}
\item[(1)]
Here we only compute the Rankin-Selberg integrals when $\tau$ is unramified because if $\tau$ is irreducible and 
generic, but not unramified, then one can show that (without the Assumption \ref{H})
$\Psi_{n,r}(v\ot\xi_s)=0$ for every $v\in\cV_\pi^{K_{n,m}}$ and $\xi_s\in I_r(\tau,s)$ (cf. \cite[Lemma 4.3]{YCheng2022}).
\item[(2)]
The proof of \eqref{E:main eqn} is resemblance to that of \cite[Equation (1.7)]{YCheng2022}, due to the similar nature and 
technical closeness of the constructions. A more interesting but difficult question is to show that 
$\Lambda_{\pi,\psi_E}(v_\pi)$ is non-zero. For this, one way is to prove that the associated Whittaker functions of 
newforms or non-zero oldforms are non-trivial on the diagonal tours. In the literature, this is usually done by using the 
$P_{n+1}$-theory, where $P_{n+1}$ is the mirobolic subgroup of $\GL_{n+1}$ 
(cf. \cite{JPSS1981}, \cite{RobertsSchmidt2007}, \cite{Miyauchi2013a}, \cite{Tsai2013}). 
Here we take a different approach which utilizes the Gan-Gross-Prasad conjecture as mentioned before. In fact, this is 
inspired by the proof of \cite[Theorem 1.1 (1)]{AtobeOiYasuda} which is also the reason for the temperedness assumption.
However, the advantage of our proof lies in its simplicity.
\item[(3)]
In this paper, we also compute the Rankin-Selberg integrals attached to oldforms under the Assumption \ref{H}. 
Notice that the formula \eqref{E:dim formula} is obtained via an abstract method; hence is not suitable for the 
computations. In order to compute these integrals, we construct conjectural bases of oldforms via level raising 
operators coming from elements in spherical Hecke algebras of $\U_{2n}(F)$ in the same spirit of 
\cite{Reeder1991} and \cite{YCheng2022}. It shall be pointed out that we need explicitly constructed oldforms to compute 
the attached Rankin-Selberg integrals, but at the same times, we also need to compute the attached Rankin-Selberg 
integrals to make sure that our constructions give non-zero oldforms. When $\pi$ is tempered, we verify that these
indeed define bases of oldforms (under the Assumption \ref{H}).
\item[(4)]
The Assumption \ref{H} should be verified through standard approaches. Namely, the Rankin-Selberg $\gamma$-factors 
should be multiplicative. 
This allows one to connect the Rankin-Selberg $\gamma$-factors to the ones defined by the Langlands-Shahidi 
method (cf. \cite{Shahidi1981}, \cite{Shahidi1990}) Next, by combining results and arguments in \cite{CPSS2011}, 
\cite{CKPSS2004} and \cite{MoeglinTadic2002}, it should be able to further relate Shahidi's 
$\gamma$-factors with those defined by the associated Weil-Deligen representations. In general, these two 
$\gamma$-factors might be equal up to an exponential, but we expect that they are actually equal in our settings.
\end{itemize}

\subsection{Organization of the paper}
This paper is organized as follows. In \S\ref{S:paramodular subgroup}, we investigate some properties of the open 
compact subgroups $K_{n,m}$ which will be used later. In \S\ref{S:double coset decomp}, we obtain double coset 
decompositions of $\U_{2n+1}(F)$ by maximal parabolic subgroups and $K_{n,m}$. 
Then we prove \thmref{T:main} in \S\ref{S:proof of main}. In \S\ref{S:basis}, we introduce conjecture bases for the spaces
of oldforms. The Rankin-Selberg integrals are reviewed in \S\ref{S:RS int}. We then compute these integrals attached to
newforms and oldforms in \S\ref{S:proof of main2}. As a result, \thmref{T:main2} is proved. Finally, in the Appendix, 
we prove a multiplicity one result for certain Hom spaces, which allows one to define Rankin-Selberg $\gamma$-factors 
attached to reducible representations.

\subsection{Notation and conventions}
Let $F$ be a finite field extension of 
$\mathbb{Q}_p$ with $p>2$ and $E$ be the unramified quadratic field extension of $F$. Denote by $x\mapsto\bar{x}$ 
the action of the non-trivial element in the Galois group ${\rm Gal}(E/F)$ on $E$. This action and its notation extend 
naturally to matrices with entries in $E$. Let 
\[
E^1=\stt{x\in E\mid x\bar{x}=1}
\]
be the group of norm one elements in $E$. Let $\frak{o}_F$ (resp. $\frak{o}_E$) be the valuation ring of $F$ (resp. $E$) 
and $\frak{p}_F$ (resp. $\frak{p}_E$) be its maximal ideal. Fix a generator $\varpi$ of $\frak{p}_F$ which is also a 
generator of $\frak{p}_E$. Denote by $k_F=\frak{o}_F/\frak{p}_F$ (resp. $k_E=\frak{o}_E/\frak{p}_E$) the residue field of 
$F$ (resp. $E$) and $q_F=q$ (resp. $q_E$) the cardinality of $k_F$ (resp. $k_E$). Then we have $q_E=q^2$.
Let $|\cdot|_E$ be the absolute value on $E$ with $|\varpi|_E=q_E^{-1}$. 
Fix an element $\delta\in\frak{o}_E^\x$ with 
$\bar{\delta}=-\delta$. Fix an unramified additive character $\psi_F$ of $F$, namely, $\psi_F$ is an additive character of 
$F$ that is trivial on $\frak{o}_F$ but not on $\frak{p}_F^{-1}$.  Put $\psi_E(x)=\psi_F(\frac{x-\b{x}}{2\delta})$ for $x\in E$. 
Then $\psi_E$ defines an unramified additive character of $E$ which is also trivial on $F$. 
Let $J_N\in\GL_N(E)$ be the element defined inductively by 
\[
J_1=(1)
\quad\text{and}\quad
J_N
=
\begin{pmatrix}
&J_{N-1}\\
1&
\end{pmatrix}
\]
and $I_N$ be the identity matrix in $\GL_N(E)$. If $a\in\GL_N(E)$, we denote $a^*=J_N{}^t\b{a}^{-1}J_N^{-1}$.\\

Suppose that $G$ is an $\ell$-group in the sense of \cite[Section 1]{BZ1976}. We denote by $\delta_G$ the modular 
function of $G$. If $K\subset G$ is an open compact subgroup, then $\sH(G//K)$ denotes the convolution algebra 
which consists of smooth compact supported bi-$K$-invariant $\bbC$-valued functions on $G$. 
In this work, by a representation of $G$ we mean a smooth complex representation with finite length. If $\pi$ is a 
representation of $G$, then its underlying (abstract) space is usually denoted by $\cV_\pi$. Finally, if $S_0$ is a subset
of a set $S$, we denote by $\mathbb{I}_{S_0}$ the characteristic function of $S_0$.

\section{Unitary groups and their open compact subgroups}\label{S:paramodular subgroup}
In this section, we set up the unitary groups and their parametrization. We then introduce the open compact 
subgroups used to define newforms and investigate some of their properties. The goal is to obtain decompositions of 
these open compact subgroups (cf. \lmref{L:decomp K}), which are important in the sequel.

\subsection{Unitary groups}
Let $\mathbb{V}_N$ be an $N$-dimensional vector space over $E$ equipped with a non-degenerated Hermitian form
\[
\langle\cdot,\cdot\rangle: \mathbb{V}_N\x\mathbb{V}_N\to E.
\]
which is $E$-linear in the first variable and satisfying $\ol{\langle v, w\rangle}=\langle w, v\rangle$ for all 
$v,w\in\mathbb{V}_N$. We assume that $\mathbb{V}_N$  admits an order basis 
\[
e_1,\ldots, e_n, v_0, f_n,\ldots, f_1
\]
whose associated Gram matrix is $J_N$. Here $n=\lfloor\frac{N}{2}\rfloor$ and we do not have $v_0$ when $N$ is even. 
Let $\U(V_N)$ be the unitary group of $V_N$, i.e. the connective reductive linear algebra group over $F$ defined by
\[
\U(\mathbb{V}_N)
=
\stt{g\in\GL(\mathbb{V}_N)\mid\text{$\langle gv, gw\rangle=\langle v,w\rangle$ for all $v,w\in V_N$}}.
\]
Notice that $\U(\mathbb{V}_N)$ is unramified over $F$ by our assumption on the ordered basis. This ordered basis also 
allows us to realize $\U(\mathbb{V}_N)$ as a subgroup $\U_N(F)$ of $\GL_N(E)$ given by
\[
\U_N(F)=\stt{g\in\GL_N(E)\mid {}^t\b{g}J_Ng=J_N}.
\]
In particular, we have $\U_1(F)=E^1$ and it can be identified as the center of $\U_N(F)$ through the embedding 
\[
x\mapsto xI_N.
\] 
To distinguish the parity of $N$, we also denote $G_n=\U_{2n+1}(F)$ and $H_r=\U_{2r}(F)$ with 
$n\ge 0$ and $r>0$ in the followings.

\subsection{Subgroups and embeddings}\label{SSS:embed}
We identify various subgroups of $\U(\mathbb{V}_{2n+1})$ with lower rank unitary groups and general linear groups via
embeddings. More specifically, if $0\le n_0\le n$ is an integer, we embed $\U({\mathbb{V}_{2n_0+1}})$ into 
$\U(\mathbb{V}_{2n+1})$ as the subgroup fixing $e_1,\ldots, e_{n-n_0}, f_{n-n_0},\ldots, f_1$. In coordinates, we have
\[
G_{n_0}\ni g_0
\mapsto
\begin{pmatrix}
I_{n-n_0}&&\\
&g_0&\\
&&I_{n-n_0}
\end{pmatrix}\in G_n.
\]
On the other hand, if $1\le r\le n$, then we identify $\U(\mathbb{V}_{2r})$ as the subgroup fixing 
$e_{r+1},\ldots, e_n, v_0, f_n,\ldots, f_{r+1}$. In coordinates, we have
\[
H_r\ni\pMX{a}{b}{c}{d}\longmapsto
\begin{pmatrix}
a&&b\\
&I_{2(n-r)+1}\\
c&&d
\end{pmatrix}\in G_n
\]
for some $a,b,c,d\in{\rm Mat}_{r\x r}(E)$. The general linear group $\GL(\mathbb{V}_r)$ embeds as the subgroup of 
$\U(\mathbb{V}_{2n+1})$ stabilizing the subspace spanned by $e_1,\ldots, e_r$ which in addition fixing the vectors 
$e_{r+1},\ldots, e_n, v_0, f_n,\ldots, f_{r+1}$. In coordinates, we have
\[
{\rm GL}_r(E)\ni a
\mapsto
\hat{a}
:=
\begin{pmatrix}
a&&\\
&I_{2(n-r)+1}&\\
&&a^*
\end{pmatrix}\in G_n
\]
where $a^*:=J_r{}^t\b{a}^{-1}J_r$. More generally, for a subset $S$ of $\GL_r(E)$, we denote 
$\hat{S}=\{\hat{a}\mid a\in S\}$.  In this paper, we do not distinguish the these groups with their images in 
$G_n$. These shall not cause serious confusions. In particular, we see that $\GL_r(E)$ is indeed contained in $H_r$.
Let $A_r\subset\GL_r(E)$ be the diagonal tours and $T_r=\hat{A}_r$. Then $T_r$ is the maximal diagonal tours of $H_r$
and the maximal diagonal tours of $G_n$ is $E^1\cdot T_n$. Finally, let $B_n\subset G_n$ be the upper triangular 
Borel subgroup and $N_n\subset B_n$ be its unipotent radical.

\subsubsection{Root subgroups}
Let $n\ge 1$ and put $N=2n+1$. We introduce certain root subgroups of $G_n$ that will be used in the later computations. 
Let $i$ be an integer between $1$ and $N$, we define $i^*=N+1-i$. Then $1\le i^*\le N$ and $(i^*)^*=i$.
For integers $1\le i, j\le N$, let $E_{ij}$ be the $N$-by-$N$ matrix whose $(i,j)$-entry is $1$ while all other entries are $0$.
Let $1\le i<j\le n$, $1\le k\le n$ and $y\in E$. We define the following elements in $G_n$:
\[
\chi_{\e_i-\e_j}(y)=I_N+yE_{ij}-\bar{y}E_{j^* i^*},
\]
\[
\chi_{\e_i+\e_j}(y)=I_N+yE_{ij^*}-\bar{y}E_{ji^*},
\]
\[
\chi_{\e_k}(y)=I_N+yE_{k,n+1}-\bar{y}E_{n+1, k^*}-2^{-1}y\bar{y}E_{k k^*}
\]
and also
\[ 
\chi_{-\e_i+\e_j}(y)={}^t\chi_{\e_i-\e_j}(y),\quad\chi_{-\e_i-\e_j}(y)={}^t\chi_{\e_i+\e_j}(y),\quad
\chi_{-\e_k}(y)={}^t\chi_{\e_k}(y).
\]
More generally, if $S\subseteq E$ is a subset and $\alpha\in\stt{\pm e_i\pm e_j, \pm e_k\mid 1\le i<j\le n, 1\le k\le n}$, 
we denote
\[
\chi_{\alpha}(S)=\{\chi_\alpha(s)\mid s\in S\}.
\]
Note that if $1\le r\le n$, then $\chi_\alpha(E)\subset H_r$ for every $\alpha\in\stt{\pm \e_i\pm\e_j\mid 1\le i<j\le r}$. 
On the other hand, we have $\chi_{\pm\e_k}(E)\cap H_r=\stt{I_N}$ for all $1\le k\le n$.

\subsubsection{Wely elements}\label{SSS:Weyl group}
Let $\sW_{G_n}$ be the spherical Weyl group of the diagonal tours $E^1\cdot T_n$ in $G_n$. Then $\sW_{G_n}$ is 
isomorphic to $\frak{S}_n\rtimes\mathbb{Z}^n_2$, where $\mathfrak{S}_n$ is the permutation group of the set 
$\cI_n:=\stt{1,2,\ldots, n}$. To obtain a set of representatives of $\sW_{G_n}$ in $G_n$, let 
$\sW_{{\rm GL}_n}\subset{\rm GL}_n(E)$ be the subgroup consisting of permutation matrices. 
For a (possibly empty) subset $S\subseteq\mathcal{I}_n$ and an element $y\in E^\times$, we put
\[
w_{S}(y)
=
\sum_{1\le i\le n,\,i\notin S} (E_{ii}+E_{i^* i^*})
+
E_{n+1,n+1}
+
\sum_{j\in S} (\bar{y}^{-1}E_{j j^*}+yE_{j^* j}).
\]
Now if we attach an element $y_S\in E^\times$ for each subset $S\subseteq\mathcal{I}_n$, then 
\begin{equation}\label{E:set of rep}
\{\hat{w}\cdot w_{S}(y_S)\mid w\in\sW_{{\rm GL}_n},\,\, S\subseteq\mathcal{I}_n\}
\end{equation}
gives a complete set of representatives of $\sW_{G_n}$ in $G_n$. Observe that this set of representatives is actually 
contained in $H_n$. In fact, the spherical Weyl group $\sW_{H_n}$ of $T_n$ in $H_n$ is also isomorphic to 
$\frak{S}_n\rtimes\mathbb{Z}^n_2$; hence \eqref{E:set of rep} also gives a complete set of representatives of 
$\sW_{H_n}$ in $H_n$. In general, if we replace $n$ in \eqref{E:set of rep} with any $1\le r\le n$, then we get a 
complete set of representatives of the spherical Weyl group $\sW_{H_r}$ of $T_r$ in $H_r$.

\subsection{Open compact subgroups}
For each integer $m\ge 0$, let $K_{n,m}$ be the open compact subgroup of $G_n$ defined by 
\begin{equation}\label{E:K_n,m}
K_{n,m} 
=
\bordermatrix{
              & n     & 1    & n     \cr
    n     & \mathfrak{o}_E &\mathfrak{o}_E &\mathfrak{p}_E^{-m}     \cr
    1     &\mathfrak{p}^m_E& 1+\mathfrak{p}^m_E &\mathfrak{o}_E\cr
    n     &\mathfrak{p}^m_E&\mathfrak{p}^m_E& \mathfrak{o}_E   \cr
            }\cap G_n.
\end{equation}
These are the open compact subgroups used to define newforms and oldforms of $G_n$ introduced by Miyauchi 
(\cite{Miyauchi2013a}) when $n=1$ and by Atobe-Oi-Yasuda (\cite{AtobeOiYasuda}) in general.  
Note that $K_{n,0}$ is the hyperspecial maximal compact subgroup of $G_n$ and
$K_{n,m}\cap G_{n_0}=K_{n_0,m}$ where $0\le n_0\le n$. Put 
\[
R_{r,m}=K_{n,m}\cap H_r.
\]
for $1\le r\le n$. Then $R_{r,0}$ and $R_{r,1}$ are non-conjugate hyperspecial maximal compact subgroups of $H_r$. 
In general, if $m=e+2\ell>1$ for some $e=0,1$ and $\ell\ge 1$, then $R_{r,m}$ is conjugate to $R_{r,e}$ by the tours 
element 
\[
\begin{pmatrix}
\varpi^\ell I_r&\\
&\varpi^{-\ell}I_r
\end{pmatrix}
\in T_r.
\] 
Let
\begin{equation}\label{E:Weyl group}
\sW_m
=
\{\hat{w}\cdot w_S(\varpi^m)\mid w\in\sW_{{\rm GL}_n},\,\,S\subseteq\mathcal{I}_n\}.
\end{equation}
Then $\sW_{G_n}$ gives a complete set of representatives of $\sW_{G_n}$ (cf. \eqref{E:set of rep}) which is 
in addition contained in $R_{n,m}$.

\begin{lm}\label{L:Iwasawa decomp}
Let $e=0,1$. We have $G_n=B_nK_{n,e}$. 
\end{lm}

\begin{proof}
When $e=0$, the assertion is the usual Iwasawa decomposition for $G_n$. So it suffices to prove the 
assertion when $e=1$. For this, we follow the proof of \cite[Lemma 2.1]{Miyauchi2013a}. The reduction of 
$K_{n,0}$ modulo $\frak{p}_E$ is isomorphic to ${\rm U}_{2n+1}(k_E/k_F)$. It then follows from the Bruhat decomposition 
of ${\rm U}_{2n+1}(k_E/k_F)$ that 
\[
G_n=B_nK_{n,0}=B_n\sW_0(K_{n,0}\cap K_{n,1}).
\]
Now since $\sW_0\subset B_n\sW_1$, we conclude that 
\[
G_n=B_n\sW_0(K_{n,0}\cap K_{n,1})=B_n\sW_1(K_{n,0}\cap K_{n,1})=B_nK_{n,1}
\]
as desired.
\end{proof}

\begin{lm}\label{L:decomp K}
Let $s_1,s_2\in\mathfrak{S}_n$ and $m\ge 1$.  Put $E^1_m=E^1\cap(1+\frak{p}_E^m)$, viewed as a subgroup of the 
center of $G_n$. Then we have
\begin{align}\label{E:decomp K}
\begin{split}
K_{n,m}
&=
E^1_m\prod_{j=1}^n\chi_{-\e_{s_2(j)}}(\frak{p}_E^m)\prod_{i=1}^{n}\chi_{\e_{s_1(i)}}(\frak{o}_E)R_{n,m}\\
&=
E^1_m\prod_{i=1}^n\chi_{\e_{s_1(i)}}(\frak{o}_E)\prod_{j=1}^{n}\chi_{-\e_{s_2(j)}}(\frak{p}^m_E)R_{n,m}.
\end{split}
\end{align}
\end{lm}

\begin{proof}
The proof of this lemma is actually similar to that of \cite[Proposition 7.1.3]{Tsai2013}. Here we only prove the second 
identity as the proof of the first one is similar. Since the subgroups $\chi_{\e_j}(\frak{o}_E)$ and 
$\chi_{-\e_i}(\frak{p}_E^m)$ are contained in $K_{n,m}$ (cf. \eqref{E:K_n,m}), it suffices to show that the LHS contains 
$K_{n,m}$. To do so, we use the observation that the stabilizer of $v_0$ in $K_{n,m}$ is $R_{n,m}$. More precisely, for 
a given $k\in K_{n,m}$, we are going to show that there exist $y_1,\ldots, y_n, z_1,\ldots, z_n$ in 
$\mathfrak{o}_E$, and $b_n\in E^1_m$ such that 
\[
b^{-1}_n\prod_{i=1}^n\chi_{\e_{s_2(n+1-i)}}(y_{n+1-i})\prod_{j=1}^n\chi_{-\e_{s_1(n+1-j)}}(\varpi^m z_{n+1-j})k v_0=v_0.
\]
Then since 
\[
b_n^{-1}\prod_{i=1}^n\chi_{\e_{s_1(n+1-i)}}(y_{n+1-i})\prod_{j=1}^n\chi_{-\e_{s_2(n+1-j)}}(\varpi^mz_{n+1-j})k\in K_{n,m}
\]
the assertion follows from the above observation.\\

Form the shape of $K_{n,m}$, we can write 
\[
kv_0
=
\sum_{i=1}^n a_i e_i+ bv_0+\sum_{j=1}^n \varpi^m c_j f_j
\]
for some $a_i, c_j$ in $\frak{o}_E$ and $b\in 1+\frak{p}_E^m$. Let $z_1\in\frak{o}_E$ to be chosen and put 
$s_2(1)=\ell$. We have
\[
\chi_{-\e_{\ell}}(\varpi^m z_1)kv_0
=
\sum_{1\le i\le n}^n a_i e_i
+
b_1v_0
+
\varpi^m(c_{\ell}-bz_1-2^{-1}a_{\ell}\varpi^m z_1\bar{z}_1)f_\ell
+
\sum_{1\le j\neq\ell\le n}\varpi^m c_j f_j
\]
where $b_1:=b+\varpi^m a_{\ell}z_1\in 1+\frak{p}_E^m$. Since $2$ and $b$ are units in $\frak{o}_E$ and $\varpi$
is a prime element of both $E$ and $F$, we can apply the proof of Hensel's lemma to conclude that there exists
$z_1\in\frak{o}_E$ such that  
\[
c_{\ell}-bz_1-2^{-1}a_{\ell}\varpi^m z_1\bar{z}_1=0.
\]
Continue this process, we get $z_1,\ldots, z_n$ in $\frak{o}_E$ such that 
\[
\prod_{j=1}^n\chi_{-\e_{s_2(n+1-j)}}(\varpi^m z_{n+1-j})kv_0
=
b_nv_0+\sum_{i=1}^n a_i e_i
\]
for some $b_n\in 1+\frak{p}_E^m$. To proceed, let $y_1\in\frak{o}_E$ to be chosen and put $s_1(1)=r$. We have
\[
\chi_{\e_r}(y_1)\prod_{j=1}^n\chi_{-\e_{s_2(n+1-j)}}(\varpi^m z_{n+1-i})kv_0
=
b_nv_0+(a_r-b_ny_1)e_r
+
\sum_{1\le i\neq r\le n} a_i e_i.
\]
Certainly, we can find $y_1\in\frak{o}_E$ so that $a_r-b_ny_1=0$. Continue this process, we can obtain 
$y_1,\ldots, y_n\in\frak{o}_E$ such that 
\[
k'v_0=b_nv_0
\quad\text{where}\quad
k':=\prod_{i=1}^n\chi_{\e_{s_1(n+1-i)}}(y_{n+1-i})\prod_{j=1}^n\chi_{-\e_{s_2(n+1-j)}}(\varpi^m z_{n+1-j})k\in K_{n,m}.
\]
It remains to show that $b_n\in E^1_m$. But this follows immediately from the following identity:
\[
b_n\bar{b}_n=\langle b_nv_0, b_nv_0\rangle=\langle k'v_0,k'v_0\rangle=\langle v_0,v_0\rangle=1.
\]
This completes the proof.
\end{proof}

\lmref{L:decomp K} has the following consequence. Let $\ell\ge 0$ be an integer and put
\[
t_\ell
=
\begin{pmatrix}
\varpi^\ell I_n&&\\
&1&\\
&&\varpi^{-\ell}I_n
\end{pmatrix}
\in
G_n.
\]
Let $m\ge 0$ be an integer and write $m=e+2\ell$ for some $e=0,1$ and $\ell\ge 0$. Define
\[
K_{n,m}^0=t_\ell K_{n,m} t^{-1}_{\ell}.
\]
Then $K^0_{n,m}=K_{n,m}$ when $m=0,1$ and \lmref{L:decomp K} implies the following decompositions for $m\ge 1$:
\begin{align}\label{E:decomp K^0}
\begin{split}
K^0_{n,m}
&=
E^1_m\prod_{j=1}^n\chi_{-\e_{s_2(j)}}(\frak{p}_E^{\ell+e})\prod_{i=1}^{n}\chi_{\e_{s_1(i)}}(\frak{p}^{\ell}_E)R_{n,e}\\
&=
E^1_m\prod_{i=1}^n\chi_{\e_{s_1(i)}}(\frak{p}_E^\ell)\prod_{j=1}^{n}\chi_{-\e_{s_2(j)}}(\frak{p}^{\ell+e}_E)R_{n,e}.
\end{split}
\end{align}
Here $s_1, s_2$ are elements in $\frak{S}_n$. In particular, we have the following filtrations:
\begin{equation}\label{E:filtration}
K^0_{n,e}\supset K^0_{n,e+2}\supset\cdots\supset K^0_{n,e+2\ell}\supset\cdots\supset R_{n,e}
=\bigcap_{m\equiv e\pmod 2} K^0_{n,m}
\end{equation}
for $e=0,1$. To investigate the double coset decompositions in the next section, we use $K_{n,m}^0$ instead of 
$K_{n,m}$.

\section{Double coset decompositions}\label{S:double coset decomp}
Let $r$ be an integer with $1\le r\le n$. Let $P_{n,r}$ be the parabolic subgroup of $G_n$ containing $B_n$ 
with the Levi decomposition $P_{n,r}=M_{n,r}\rtimes N_{n,r}$, where $M_{n,r}=\widehat{{\rm GL}}_r(E)\times G_{n-r}$
(cf. \S\ref{SSS:embed}) and $N_{n,r}\subset N_n$. Let $\b{P}_{n,r}\supset M_{n,r}$ be the opposite of $P_{n,r}$. 
We first investigate the double coset decomposition:
\[
\bar{P}_{n,r}\backslash G_n/K^0_{n,m}.
\]

\begin{lm}\label{L:representative}
Let $m\ge 0$ be an integer and write $m=e+2\ell$ for some $e=0,1$ and $\ell\ge 0$. Then the set
\[
\stt{\chi_{\e_i}(\varpi^{d})\mid 0\le d\le\ell}
\]
gives a complete set of representatives of $\bar{P}_{n,r}\backslash G_n/K^0_{n,m}$. 
Here $i$ is any integer between $1$ and $r$.
\end{lm}

\begin{proof}
It suffices to prove the lemma for $i=1$ since $\chi_{\e_i}(y)$ and $\chi_{\e_1}(y)$ are conjugate by an element in 
$\widehat{\sW}_{\GL_r}$ for $1\le i\le r$, and $\widehat{\sW}_{\GL_r}$ (cf. \S\ref{SSS:Weyl group}) is contained in both 
$\bar{P}_{n,r}$ and $K^0_{n,m}$. We begin with the case when $m=0,1$. Recall that $K_{n,m}^0=K_{n,m}$ and note that
$e=m$ and $\ell=0$ in this case. Since $w_{\mathcal{I}_n}(\varpi^e)\in K_{n,e}$ and 
$\bar{B}_n=w_{\cI_n}(\varpi^e)B_n w_{\cI_n}(\varpi^e)^{-1}$ 
we get from \lmref{L:Iwasawa decomp} that 
\begin{equation}\label{E:general Iwasawa decomp}
G_n
=
w_{\cI_n}(\varpi^e)G_n w_{\cI_n}(\varpi^e)^{-1}
=
w_{\cI_n}(\varpi^e)B_nK_{n,e} w_{\cI_n}(\varpi^e)^{-1}
=
\bar{B}_nK_{n,e}
=
\bar{P}_{n,r}K_{n,e}.
\end{equation}
As $\chi_{\e_1}(1)\in K_{n,e}$, the case $m=0,1$ is proved.\\

Suppose from now on that $m\ge 2$ so that $\ell\ge 1$. We split the proof into three steps. The first step is to show that 
the set 
\begin{equation}\label{E:1st rep}
\stt{\prod_{i=1}^r\chi_{\e_i}(\varpi^{d_i})\mid 0\le d_1\le\cdots\le d_r\le\ell}
\end{equation}
contains a complete set of representatives of $\b{P}_{n,r}\backslash G_n/K^0_{n,m}$. 
Observe that by \eqref{E:general Iwasawa decomp}, 
\[
\b{P}_{n,r}\backslash G_n/K^0_{n,m}
=
(\bar{P}_{n,r}\cap K_{n,e})\backslash K_{n,e}/K^0_{n,m}.
\]
We therefore separate the case $e=0$ from $e=1$. Assume first that $e=1$. Then we have the decomposition
\[
K_{n,1}
=
E_1^1\prod_{j=1}^n\chi_{-\e_j}(\frak{p}_E)
\prod_{i=r+1}^n\chi_{\e_i}(\frak{o}_E)
\prod_{i=1}^r\chi_{\e_i}(\frak{o}_E)
R_{n,1}
\]
by \lmref{L:decomp K}. Since the subgroups $\chi_{\e_i}(\frak{p}_E^\ell)$ for $1\le i\le r$ and $R_{n,1}$ are contained in 
$K^0_{n,m}$ (cf. \eqref{E:decomp K^0}); while the subgroups $E_1^1$, $\chi_{\e_{-j}}(\frak{p}_E)$ for $1\le j\le n$ and 
$\chi_{\e_j}(\frak{o}_E)$ for $r+1\le j\le n$ are contained in $\bar{P}_{n,r}\cap K_{n,1}$, we see that 
\[
\stt{\prod_{i=1}^r\chi_{\e_i}(\varpi^{d_i}y_i)\mid \text{$d_i\ge 0$, $y_i\in\frak{o}_E^\x$ for $1\le i\le r$ and one of 
$d_i\le\ell$}}
\]
contains a complete set of representatives of the double coset. 
Because we are allowing to conjugate elements in the subgroup
\[
T_n\cap R_{n,1}
=
\stt{{\rm diag}(y_1,\ldots, y_n, 1, \bar{y}_n^{-1},\ldots,\bar{y}_1^{-1})\mid y_1,\ldots,y_n\in \frak{o}_E^\x}
\]
a complete set of representatives can be chosen in the set
\begin{equation}\label{E:rough rep}
\stt{\chi(d_1,\ldots, d_r):=\prod_{i=1}^r\chi_{\e_i}(\varpi^{d_i})\mid\text{$d_i\ge 0$ for $1\le i\le r$ and one of $d_i\le\ell$}}.
\end{equation}
To reduce further, we conjugate elements in $\widehat{\sW}_{{\rm GL}_r}$.
If $w\in\sW_{{\rm GL}_r}$ is the element corresponding to $s\in\frak{S}_r$, then 
\[
\hat{w}\cdot \chi(d_1,\ldots, d_r)\cdot\hat{w}^{-1}=\chi(d_{s(1)},\ldots, d_{s(r)}).
\]
In particular,  we may in addition assume that $d_1\le\cdots\le d_r$ in \eqref{E:rough rep}. Since
\[
\bar{P}_{n,r}\chi(d_1,\ldots,d_r)K^0_{n,m}
=
\bar{P}_{n,r}\prod_{i=1}^j\chi_{\e_i}(\chi^{d_i}) K^0_{n,m}
=
\bar{P}_{n,r}\chi(d_1,\ldots, d_j,\underbrace{\ell,\ldots,\ell}_{r-j})K^0_{n,m}.
\]
if there exists $0\le  j< r$ such that $d_j<\ell\le d_{j+1}$ (with $d_0:=0$), 
our assertion for the case $e=1$ follows.\\

Suppose that $e=0$. In this case, we don't have an analogue decomposition for $K_{n,0}$, so we use the 
following decomposition 
\[
K_{n,0}=(\bar{N}_n\cap K_{n,0})(N_n\cap K_{n,0})(E^1\cdot T_n\cap K_{n,0})\sW_0
\]
instead, where $\bar{N}_n$ is the unipotent radical of $\bar{B}_n$. Recall that $\sW_0\subset R_{n,0}$
is given by \eqref{E:Weyl group}. 
Clearly, we have 
\[
\bar{N}_n\cap K_{n,0}\subset\bar{P}_{n,r}\cap K_{n,0}
\quad\text{and}\quad
(E^1\cdot T_n\cap K_{n,0})\sW_0\subset K_{n,0}.
\] 
On the other hand, since 
\[
(N_n\cap K_{n,0})
=
\left(\prod_{i=r+1}^n\chi_{\e_i}(\frak{o}_E)\right)
\left(\prod_{i=1}^r\chi_{\e_i}(\frak{o}_E)\right)
(N_n\cap R_{n,0})
\]
and the subgroups $\chi_{\e_i}(\frak{p}_E^\ell)$ for $1\le i\le r$ and $R_{n,0}$ are contained in $K^0_{n,m}$ by 
\eqref{E:decomp K^0}, and moreover, the subgroups $\chi_{\e_i}(\frak{o}_E)$ for $r+1\le i\le n$ are contained in 
$\bar{P}_{n,r}\cap K_{n,0}$, we are reducing to the similar situation as in the case $e=1$. 
Now the same argument prove the assertion for the case $e=0$. This completes the first step.\\

The next step is to show that 
\[
\stt{\chi_{\e_1}(\varpi^{d})\mid 0\le d\le\ell}
\]
contains a complete set of representatives of 
$\bar{P}_{n,r}\backslash G_n/K^0_{n,m}$. Since the assertion for $r=1$ is already proved in the first step, we may assume 
$r>1$. To proceed, the following identity is needed
\begin{equation}\label{E:id}
\chi_{\e_{k-1}}(\varpi^{d'})\chi_{\e_k}(\varpi^{d})
=
\chi_{-\e_{k-1}+\e_k}(\varpi^{d-d'})
\chi_{\e_{k-1}}(\varpi^{d'})
\chi_{-\e_{k-1}+\e_k}(-\varpi^{d-d'})
\chi_{\e_{k-1}+\e_k}(2^{-1}\varpi^{d+d'})
\end{equation}
for $1<k\le n$ and $d,d'\in\mathbb{Z}$. Observe that if $k\le r$ and $d'\le d$, then 
$\chi_{-\e_{k-1}+\e_k}(\varpi^{d-d'})\in\bar{P}_{n,r}$ and both of 
$\chi_{-\e_{k-1}+\e_k}(-\varpi^{d-d'})$ and $\chi_{\e_{k-1}+\e_k}(2^{-1}\varpi^{d+d'})$ are contained in $R_{n,m}$.
Now let 
\[
g
=
\prod_{i=1}^r\chi_{\e_i}(\varpi^{d_i})
=
g_0\chi_{\e_{r-1}}(\varpi^{d_{r-1}})\chi_{\e_r}(\varpi^{d_r})
\] 
be an element in the set given by \eqref{E:1st rep}, where $g_0=\prod_{i=1}^{r-2}\chi_{\e_i}(\varpi^{d_i})$  and we define
$g_0=I_{2n+1}$ when $r=2$.  Applying \eqref{E:id} and noting that $\chi_{-\e_{r-1}+\e_r}(y)$ commutes with 
$\chi_{\e_i}(z)$ for every $1\le i\le r-2$, we get that 
\begin{align*}
\bar{P}_{n,r}gK^0_{n,m}
&=
\bar{P}_{n,r}g_0
\chi_{-\e_{r-1}+\e_r}(\varpi^{d_r-d_{r-1}})
\chi_{\e_{r-1}}(\varpi^{d_{r-1}})
\chi_{-\e_{r-1}+\e_r}(-\varpi^{d_r-d_{r-1}})
\chi_{\e_{r-1}+\e_r}(2^{-1}\varpi^{d_r+d_{r-1}})K^0_{n,m}\\
&=
\bar{P}_{n,r}
\chi_{-\e_{r-1}+\e_r}(\varpi^{d_r-d_{r-1}})
g_0
\chi_{\e_{r-1}}(\varpi^{d_{r-1}})K^0_{n,m}\\
&=
\bar{P}_{n,r}\prod_{i=1}^{r-1}\chi_{\e_i}(\varpi^{d_i})K^0_{n,m}.
\end{align*}
The assertion then follows from continuing this process.
This finishes the second step.\\

The final step is to show that the double cosets 
\[
\bar{P}_{n,r}\chi_{\e_1}(\varpi^d)K^0_{n,m}
\]
for $0\le d\le \ell$ are all distinct. 
Suppose in contrast that there exist $0\le d'<d\le\ell$ and $h\in\bar{P}_{n,r}$, $k\in K^0_{n,m}$ such that 
\begin{equation}\label{E:d=c}
\chi_{\e_1}(\varpi^d)
=
h\chi_{\e_1}(\varpi^{d'})k.
\end{equation}
To obtain a contradiction, note that 
\[
h=\chi_{\e_1}(\varpi^{d})k^{-1}\chi_{\e_1}(-\varpi^{d'})
\in
K^0_{n,m}\cap\bar{P}_{n,r}
\]
implies
\begin{equation}\label{E:hf_j}
h^{-1}f_i
=
\sum_{j=1}^r a_{ij}f_j,
\end{equation}
for some $a_{ij}\in\frak{o}_E$ for $1\le i, j\le r$. Also, from the shape of $K_{n,m}^0$, we can write
\[
kv_0
=
\sum_{i=1}^n \varpi^\ell a_ie_i +bv_0+\sum_{j=1}^n \varpi^{\ell+e} c_j f_j
\]
for some $a_i, c_j$ in $\frak{o}_E$ and $b\in 1+\frak{p}_E^m$. Then we have
\begin{equation}\label{E:kv_0}
\chi_{\e_1}(\varpi^{d'})kv_0
=
\varpi^{d'} b'e_1
+
\sum_{i=2}^n
\varpi^\ell a_i e_i
+
(b-\varpi^{d'} c_1)v_0
+
\sum_{j=1}^n\varpi^{\ell+e}f_j
\end{equation}
where $b':=b+\varpi^{\ell-d'} a_1-2^{-1}\varpi^{d'+\ell+e}c_1\in\frak{o}_E^\times$.
Now let $1\le i\le r$. We are going to compute $\langle\chi_{\e_1}(\varpi^d)v_0, f_i\rangle$ in two ways. On one hand, 
\[
\langle\chi_{\e_1}(\varpi^d)v_0, f_i\rangle
=
\langle v_0+\varpi^d e_1, f_i\rangle
=
\varpi^d\delta_{1i}.
\]
On the other hand, from \eqref{E:d=c}, \eqref{E:hf_j} and \eqref{E:kv_0}, 
\[
\langle\chi_{\e_1}(\varpi^d)v_0, f_i\rangle
=
\langle\chi_{\e_1}(\varpi^{d'})kv_0, h^{-1}f_i\rangle
=
\varpi^{d'}b'\bar{a}_{i1}+\sum_{j=2}^n\varpi^\ell a_j\bar{a}_{ij}.
\]
We thus obtain
\[
\varpi^{d'}b'\bar{a}_{i1}+\sum_{j=2}^n\varpi^\ell a_j\bar{a}_{ij}=\varpi^d\delta_{1i}.
\]
From this identity, we see that $a_{i1}\in\frak{p}_E$ for $2\le i\le n$ since $d'<\ell$. 
It then follows that $a_{11}\in\frak{o}_E^\x$ as $h$ (and hence $h^{-1}$) is contained in $K^0_{n,m}\cap\bar{Q}_{n,r}$ 
(cf. \eqref{E:hf_j}). But this would imply
\[
\varpi^d
=
\varpi^{d'}b'\bar{a}_{11}+\sum_{j=2}^n\varpi^\ell a_j\bar{a}_{1j}\in\varpi^{d'}\frak{o}_E^\times
\]
which contradicts to the assumption $d'<d$. This concludes the final step and also the proof of 
\lmref{L:representative}.
\end{proof}

To state the next lemma, we need some notation. For a given $h\in\bar{P}_{n,r}$, we denote by 
$a_h\in M_{n,r}$ the "Levi part" of $h$ under the Levi decomposition $\bar{P}_{n,r}=M_{n,r}\rtimes\bar{N}_{n,r}$.
On the other hand, we let $\Gamma^0_{r,m}\subseteq{\rm GL}_r(\frak{o}_E)$ be the open compact subgroup 
defined by
\[
\Gamma'_{r,m}
=
\bordermatrix{
              & (r-1)     & 1       \cr
    (r-1) & \mathfrak{o}_E &\mathfrak{p}^m_E    \cr
    1     &\mathfrak{o}_E& 1+\mathfrak{p}^m_E \cr
            }\cap {\rm GL}_r(\frak{o}_E).
\]

\begin{lm}\label{L:Levi}
Let $m=e+2\ell\ge 0$ be an integer with $e=0,1$ and $\ell\ge 0$. Let $r, d$ be integers with $1\le r\le n$ and 
$0\le d\le \ell$. Let $\b{M}^d_{n,r}\subset M_{n,r}$ be the subgroup defined by 
\[
\bar{M}^d_{n,r}
=
\stt{a_h\mid h\in\chi_{\e_r}(\varpi^d)K^0_{n,m}\chi_{\e_r}(\varpi^d)^{-1}\cap\bar{P}_{n,r}}.
\]
Then we have $\bar{M}^d_{n,r}=\widehat{\Gamma}'_{r,\ell-d}\times K^0_{n-r,e+2d}$.
\end{lm}

\begin{proof}
For convenient, we also denote $s_{r,d}=\chi_{\e_r}(\varpi^d)$, $U_E^0=\frak{o}_E^\x$ and 
$U_E^j=1+\frak{p}_E^j\subset\frak{o}_E^\x$ for $j>0$. Note that the lemma holds trivially when $m=0,1$ since in this case,
$\ell=d=0$ and $s_{r,0}\in K^0_{n,m}=K_{n,m}$. So we assume in the rest of the proof that $m\ge 2$ (so that $\ell\ge 1$). 
We start with the case $r=1$. The first step is to show 
\begin{equation}\label{E:r=1 contain}
\widehat{U}_E^{\ell-d}\times K^0_{n-1,e+2d}\subseteq\bar{M}^d_{n,1}.
\end{equation}
To do so, for a given $a$ in the LHS of \eqref{E:r=1 contain}, our strategy is to find $h\in\b{P}_{n,1}$ with $a_h=a$ and 
$s_{1,d}^{-1}hs_{1,d}\in K^0_{n,m}$. This would imply $a\in\b{M}^d_{n,1}$.  Now let us write 
\[
s_{1,d}
=
\begin{pmatrix}
1&\alpha&\beta\\
&1_{2n-1}&\alpha'\\
&&1
\end{pmatrix}
\]
with 
\[
\alpha
=
(\underbrace{0,\ldots,0}_{n-1}, \varpi^d, \underbrace{0,\ldots,0}_{n-1}),
\quad
\alpha'
=
-{}^t\alpha
\quad\text{and}\quad
\beta=-2^{-1}\varpi^{2d}.
\]
Then
\[
s^{-1}_{1,d}
=
\begin{pmatrix}
1&-\alpha&\beta\\
&1_{2n-1}&-\alpha'\\
&&1
\end{pmatrix}.
\]
Let 
\[
h'
=
\begin{pmatrix}
a&&\\
&h&\\
&&\bar{a}^{-1}
\end{pmatrix}
\in
\bar{P}_{n,1}
\]
with $a\in U_E^{\ell-d}$ and $h\in R_{n-1,e}$. Then a direct computation shows 
\[
s_{1,d}^{-1}h's_{1,d}
=
\begin{pmatrix}
a&a\alpha-\alpha h& a\beta-\alpha h\alpha'+\beta\bar{a}^{-1}\\
&h&h\alpha'-\alpha'\bar{a}^{-1}\\
&&\bar{a}^{-1}
\end{pmatrix}
=
\begin{pmatrix}
a&(a-1)\alpha& a\beta+\beta\bar{a}^{-1}+\varpi^{2d}\\
&h&\alpha'(1-\bar{a}^{-1})\\
&&\bar{a}^{-1}
\end{pmatrix}
\in
K^0_{n,m}.
\]
It follows that 
\begin{equation}\label{E:1 contain}
\widehat{U}_E^{\ell-d}\x R_{n-1,e}\subseteq\b{M}^d_{n,1}.
\end{equation}
To establish \eqref{E:r=1 contain}, it remains to show that $\hat{1}\x E^1_m$,  $\hat{1}\x\chi_{\e_{j}}(\frak{p}^d_E)$ and 
$\hat{1}\x\chi_{-\e_{j}}(\frak{p}^{e+d})$ are contained in $\b{M}^d_{n,1}$ for $2\le j\le n$. Here we regard $E_m^1$ as 
a subgroup in the center of $G_{n-1}$ and we define $E^1_0=E^1$. We check this case by case. First let $y\in E^1_m$. 
Then
\[
\begin{pmatrix}
1&&\\
&y I_{2n-1}&\\
&&1
\end{pmatrix}
\in \bar{P}_{n,1}
\]
and we have
\[
s_{1,d}^{-1}
\begin{pmatrix}
1&&\\
&y I_{2n-1}&\\
&&1
\end{pmatrix}
s_{1,d}
=
\begin{pmatrix}
1&\alpha(1-y)&2\beta-\alpha y\alpha'\\
&y I_{2n-1}&\alpha'(y-1)\\
&&1
\end{pmatrix}
\in K^0_{n,m}.
\]
This gives $\hat{1}\x E_m^1\subset \b{M}_{n,1}^d$.  Next let $y\in\frak{o}_E$ and $2\le j\le n$. Then 
\[
\chi_{-\e_1+\e_{j}}(-y)\chi_{\e_{j}}(\varpi^i y)\in\bar{P}_{n,1}
\]
and the "Levi part" of $\chi_{-\e_1+\e_{j}}(-y)\chi_{\e_{j}}(\varpi^d y)$ is $\chi_{\e_j}(\varpi^d y)$. Furthermore, if we write 
\[
\chi_{-\e_1+\e_{j}}(-y)
=
\begin{pmatrix}
1&&\\
x&I_{2n-1}&\\
0&x'&1
\end{pmatrix}
\quad\text{and}\quad
\chi_{\e_{j}}(\varpi^d y)
=
\begin{pmatrix}
1&&\\
&h&\\
&&1
\end{pmatrix}
\]
with 
\[
{}^tx=(\underbrace{0,\ldots,0}_{j-2},-y,\underbrace{0,\ldots,0}_{2n-j}),
\quad
x'=-{}^t\bar{x}J_{2n-1}
\quad\text{and}\quad
h\in G_{n-1}.
\]
then 
\begin{align*}
s_{1,d}^{-1}\chi_{-\e_1+\e_{j}}(-y)\chi_{\e_j}(\varpi^d y)s_{1,d}
=
\begin{pmatrix}
1&\alpha-\alpha h+\beta x' h & 2\beta-\alpha h\alpha'+\beta x' h\alpha'\\
x&x\alpha+h-\alpha'x' h&x\beta+h\alpha'-\alpha' x' h\alpha'-\alpha'\\
0&x'h &1+x' h\alpha'
\end{pmatrix}
\in K_{n,m}^0.
\end{align*}
It follows that $\hat{1}\x\chi_{\e_j}(\frak{p}^d_E)\subset\b{M}^d_{n,1}$. Similarly, if $y\in\frak{o}_E$ and 
$2\le j\le n$, then 
\[
\chi_{-\e_1-\e_{j}}(\varpi^e\bar{y})\chi_{-\e_{j}}(\varpi^{e+d} y)\in\bar{P}_{n,1}
\]
and we have 
\[
s_{1,d}^{-1}\chi_{-\e_1-\e_{j}}(\varpi^e\bar{y})\chi_{-\e_{j}}(\varpi^{e+d} y)s_{1,d}
\in
K^0_{n,m}.
\]
Thus we also have $\hat{1}\x\chi_{-\e_{j}}(\frak{p}^{e+d}_E)\subset\bar{M}^d_{n,1}$. 
Now \eqref{E:r=1 contain} follows from these and \eqref{E:1 contain}.
Indeed, if $e+2d\ge 1$, then this is a consequence of \eqref{E:decomp K^0}. On the other hand, if $e=d=0$, then we use
the facts that $K_{n-1,0}$ is generated by the subgroups $E^1$, $R_{n-1,0}$ and $\chi_{\pm\e_j}(\frak{o}_E)$ for 
$2\le j\le n$.\\

To prove the reverse inclusion, first note that $s_{1,d}\in K^0_{n, e+2d}$ and $K^0_{n,m}\subseteq K^0_{n,e+2d}$.
It follows that 
\[
s_{1,d}K^0_{n,m}s_{1,d}^{-1}\cap\bar{P}_{n,1}\subseteq K^0_{n,e+2d}\cap\bar{P}_{n,1}.
\]
and hence 
\[
\bar{M}^d_{n,1}
\subseteq
\hat{\frak{o}}^\x_E\x K^0_{n-1,e+2d}.
\]
Thus it remains to show that if $s_{1,d}^{-1}h s_{1,d}\in K^0_{n,m}$, where
\[
h
=
\begin{pmatrix}
a&&\\
&h'&\\
&&\bar{a}^{-1}
\end{pmatrix}
\begin{pmatrix}
1&&\\
x&1_{2n-1}&\\
y&x'&1
\end{pmatrix}
\in
\bar{P}_{n,1}
\]
then $a\in U_E^{\ell-d}$. For this, let us compute
\begin{align*}
\langle s_{1,d}^{-1}hs_{1,d}v_0, f_1\rangle
&=
\langle h s_{1,d}v_0, s_{1,d}f_1\rangle\\
&=
\langle
h s_{1,d}v_0, f_1-\varpi^iv_0-2^{-1}\varpi^{2d}e_1
\rangle\\
&=
\langle
s_{1,d}v_0, h^{-1}f_1 
\rangle
-
\langle
hs_{1,d}v_0, \varpi^dv_0+2^{-1}\varpi^{2d}e_1
\rangle\\
&=
\langle
v_0+\varpi^d e_1, \bar{a} f_1
\rangle
-
\varpi^d
\langle
h s_{d,1}v_0, v_0+\varpi^d e_1
\rangle
+
2^{-1}\varpi^{2d}
\langle
h s_{1,d} v_0, e_1
\rangle\\
&=
\varpi^d a
-
\varpi^d
\langle
hs_{1,d}v_0, s_{1,d}v_0
\rangle
+
2^{-1}\varpi^{2d}
\langle
h s_{1,d} v_0, s_{1,d}e_1
\rangle\\
&=
\varpi^d a
-
\varpi^d
\langle
s_{1,d}^{-1}h s_{1,d}v_0, v_0
\rangle
+
2^{-1}\varpi^{2d}
\langle
s^{-1}_{1,d} h s_{1,d} v_0, e_1
\rangle.
\end{align*}
Here we use the facts that $h^{-1}f_1=\bar{a}$ and $s_{1,d}e_1=e_1$.
Since $s_{1,d}^{-1}h s_{1,d}\in K^0_{n,m}$, the shape of $K^0_{n,m}$ implies that
\[
\langle s_{1,d}^{-1}hs_{1,d}v_0, f_1\rangle\in\frak{p}_E^\ell,
\quad
\langle s^{-1}_{1,d}hs_{1,d}v_0, v_0\rangle\in U_E^m
\quad\text{and}\quad
\langle
s^{-1}_{1,d} h s_{1,d} v_0, e_1
\rangle\in\frak{p}_E^{e+\ell}.
\]
From these we conclude $a\in U_E^{\ell-d}$. This proves the lemma when $r=1$.\\

Suppose now that $r>1$. The proof of this case is similar to that of $r=1$. In fact, we will apply the result for $r=1$. 
Again, we first establish the following inclusion
\begin{equation}\label{E:contain 5}
\widehat{\Gamma}'_{r,\ell-d}\times K^0_{n-r,e+2d}
\subseteq
\bar{M}^d_{n,r}
\end{equation}
which is clearly a consequence of 
\begin{equation}\label{E:2 contain}
\hat{1}_r\x K^0_{n-r,e+2d}\subset\bar{M}_{n,r}^d
\end{equation}
and
\begin{equation}\label{E:3 contain}
\widehat{\Gamma}'_{r,\ell-d}\x I_{2(n-r)+1}\subset\bar{M}^d_{n,r}.
\end{equation}
To prove \eqref{E:2 contain}, note that $s_{r,d}\in G_{n-r+1}$ (recall the embedding $G_{n-r+1}\hookto G_n$ 
(cf. \S\ref{SSS:embed})) and hence 
\[
s_{r,d}K^0_{n-r+1, m}s_{r,d}^{-1}\subset G_{n-r+1}.
\] 
Since $G_{n-r+1}\cap\bar{P}_{n,r}=\bar{P}_{n-r+1,1}$, we find that 
\[
s_{r,d}K^0_{n-r+1, m}s_{r,d}^{-1}\cap\bar{P}_{n,r}
=
s_{r,d}K^0_{n-r+1, m}s_{r,d}^{-1}
\cap
G_{n-r+1}
\cap
\bar{P}_{n,r}
=
s_{r,d}K^0_{n-r+1, m}s_{r,d}^{-1}\cap\bar{P}_{n-r+1,1}.
\]
Now we can apply the result for $r=1$ (with $n$ replaced by $n-r+1$) to obtain \eqref{E:2 contain}. 
Next, we show \eqref{E:3 contain}. Let $a\in\Gamma'_{r,\ell-d}$ and write
\[
s_{r,d}
=
\begin{pmatrix}
1_n&\alpha&\beta\\
&1&\alpha'\\
&&1_n
\end{pmatrix}
\quad\text{and}\quad
\hat{a}
=
\begin{pmatrix}
A&&\\
&1&\\
&&A^*
\end{pmatrix}
\]
with 
\[
{}^t\alpha
=
(\underbrace{0,\ldots, 0}_{r-1},\varpi^d,\underbrace{0,\ldots,0}_{n-r}),
\quad
\alpha'=-{}^t\alpha J_n,
\quad
\beta
=
-2^{-1}\varpi^{2d}E_{rr}\in{\rm Mat}_{n\x n}(E)
\quad\text{and}\quad 
A
=
\begin{pmatrix}
a&\\
&I_{n-r}
\end{pmatrix}.
\]
Then we have
\[
s_{r,d}^{-1}\hat{a}s_{r,d}
=
\begin{pmatrix}
A&A\alpha-\alpha&A\beta-\alpha\alpha'+\beta A^*\\
&1&\alpha'-\alpha'A^*\\
&&A^*
\end{pmatrix}
\in
K^0_{n,m}.
\]
This shows \eqref{E:3 contain} and hence \eqref{E:contain 5}.\\

Now we prove the reverse inclusion. First note that since $s_{r,d}\in K^0_{n,e+2d}$ and 
$K^0_{n,m}\subseteq K^0_{n,e+2d}$, we have
\[
s_{r,d}K^0_{n,m}s_{r,d}^{-1}\cap\bar{P}_{n,r}\subseteq K^0_{n,e+2d}\cap\bar{P}_{n,r}.
\]
This implies
\begin{equation}\label{E:contain 4}
\bar{M}_{n,r}^d\subseteq\widehat{\GL}_r(\frak{o}_E)\x K^0_{n-r,e+2d}.
\end{equation}
In particular, the case $d=\ell$ is proved since $\Gamma'_{r,0}={\rm GL}_r(\frak{o}_E)$ and we have \eqref{E:contain 5}. 
So assume $d<\ell$ and let $h\in\bar{P}_{n,r}$ such that $s^{-1}_{r,d}hs_{r,d}\in K^0_{n,m}$. Then \eqref{E:contain 4} 
implies
\[
h^{-1}f_j
=
\sum_{k=1}^r a_{jk} f_k
\]
for some $a_{jk}\in\frak{o}_E$ with $1\le j, k\le r$. We need to show that $a_{rr}\in 1+\frak{p}_E^{\ell-d}$ and 
$a_{jr}\in\frak{p}_E^{\ell-d}$ for $1\le j\le r-1$. The idea of which is similar to the case $r=1$.
Form the shape of $K^0_{n,m}$, we see that 
\[
\langle
s^{-1}_{r,d}hs_{r,d} v_0, f_j
\rangle
\in
\frak{p}_E^\ell
\] 
for $1\le j\le r$. If $1\le j\le r-1$, then $s_{r,d}f_j=f_j$, and hence
\begin{align*}
\langle
s^{-1}_{r,d}hs_{r,d} v_0, f_j
\rangle
=
\langle
hs_{r,d}v_0, f_j
\rangle
=
\langle
s_{r,d}v_0, h^{-1}f_j
\rangle
=
\sum_{k=1}^r
\langle
v_0+\varpi^d e_r, a_{jk}f_k
\rangle
=
\varpi^d\bar{a}_{jr}.
\end{align*}
These imply $a_{jr}\in\frak{p}_E^{\ell-d}$ for $1\le j\le r-1$. If $j=r$, then
\begin{align*}
\langle
s^{-1}_{r,d}hs_{r,d} v_0, f_r
\rangle
&=
\langle
hs_{r,d}v_0,  f_r-\varpi^d v_0-2^{-1}\varpi^{2d}e_r
\rangle\\
&=
\langle
s_{r,d}v_0, h^{-1}f_r
\rangle
-
\varpi^d
\langle
hs_{r,d}v_0, v_0+2^{-1}\varpi^d e_r
\rangle\\
&=
\sum_{k=1}^r
\langle
v_0+\varpi^d e_r, a_{rk}f_k 
\rangle
-
\varpi^d
\langle
hs_{r,d}v_0, v_0+\varpi^d e_r
\rangle
+
2^{-1}\varpi^{2d}
\langle
hs_{r,d} v_0, e_r
\rangle\\
&=
\varpi^d\bar{a}_{rr}
-
\varpi^d
\langle
hs_{r,d}v_0, s_{r,d}v_0
\rangle
+
2^{-1}\varpi^{2d}
\langle
hs_{r,d}v_0, s_{r,d}e_r
\rangle\\
&=
\varpi^i\bar{a}_{rr}
-
\varpi^d
\langle
s^{-1}_{r,d}hs_{r,d}v_0, v_0
\rangle
+
2^{-1}\varpi^{2d}
\langle
s^{-1}_{r,d}hs_{r,d}v_0, e_r
\rangle.
\end{align*}
Since (again from the shape of $K^0_{n,m}$)
\[
\langle
s^{-1}_{r,d}hs_{r,d}v_0, v_0
\rangle
\in 1+\frak{p}_E^m
\quad\text{and}\quad
\langle
s^{-1}_{r,d}hs_{r,d} v_0, e_r
\rangle
\in
\frak{p}_E^{e+\ell}
\]
we conclude that $\bar{a}_{rr}\in 1+\frak{p}_E^{\ell-d}$ and hence $a_{rr}\in 1+\frak{p}_E^{\ell-d}$ as desired. 
This finishes the proof.
\end{proof}

Let $\Gamma_{r,m}\subseteq{\rm GL}_r(\frak{o}_E)$ be the usual "congruence 
subgroup", namely, 
\[
\Gamma_{r,m}
=
\bordermatrix{
              & (r-1)     & 1       \cr
    (r-1) & \mathfrak{o}_E &\mathfrak{o}_E    \cr
    1     &\mathfrak{p}^m_E& 1+\mathfrak{p}^m_E \cr
            }\cap {\rm GL}_r(\frak{o}_E).
\]
Then we have the following corollary:

\begin{cor}\label{C:inter Levi}
Let $r, m$ be integers with $1\le r\le n$ and $m\ge 0$. Write $m=e+2\ell$ for some $e=0,1$ and $\ell\ge 0$.
Then for any $1\le j\le r$, the set 
\[
\stt{\chi_{-\e_j}(\varpi^{e+d})\mid 0\le d\le\ell}
\]
forms a complete set of representatives of $P_{n,r}\backslash G_n/ K^0_{n,m}$. Moreover, we have
\[
M^d_{n,r}
:=
\stt{a_h\mid h\in\chi_{-\e_r}(\varpi^{e+d})K^0_{n,m}\chi_{-\e_r}(\varpi^{e+d})^{-1}\cap P_{n,r}}
=
\hat{\Gamma}_{r,\ell-d}\x K^0_{n-r,e+2d}
\]
where $a_h\in M_{n,r}$ denotes the "Levi part" of $h\in P_{n,r}$ under the Levi decomposition 
$P_{n,r}=M_{n,r}\ltimes N_{n,r}$.
\end{cor}

\begin{proof}
Let $S=\{1,2,\ldots, r\}\subseteq\mathcal{I}_n$ and $w=w_{S}(\varpi^e)$ (cf. \S\ref{SSS:Weyl group}).
Then $w\in R_{n,e}\subset K^0_{m,n}$ and we have $w\bar{P}_{n,r}w^{-1}=P_{n,r}$ with 
\[
w\,\diag{a, g_0, a^*}\, w^{-1}=\diag{{}^t\b{a}^{-1}, g_0, J_raJ_r^{-1}}
\] 
for $\diag{a,g_0, a^*}\in M_{n,r}$, where $a\in\GL_r(E)$ and $g_0\in G_{n-r}$. Since the map 
$a\mapsto{}^t\bar{a}^{-1}$ gives an isomorphism from $\Gamma'_{r,m}$ onto $\Gamma_{r,m}$ and we have 
$w\chi_{\e_j}(\varpi^d)w^{-1}=\chi_{-\e_j}(\varpi^{e+d})$ for $1\le j\le r$, the corollary follows immediately
from \lmref{L:representative} and \lmref{L:Levi}.
\end{proof}

\section{Proof of \thmref{T:main}}\label{S:proof of main}
Let us proof \thmref{T:main} in this section. To begin with, we need some preparations.

\subsection{Preliminaries}\label{SS:pre1}
The following lemma describe a property of the local Langlands correspondence for $\GL_r$, which should be well-known 
to the experts. However, since we can't locate a proper reference, we shall provide a proof here. To state the 
lemma, let $\sigma$ be an irreducible representation of $\GL_N(E)$ with the associated $L$-parameter 
$\phi_\sigma:WD_E\to\GL_N(\mathbb{C})$ (cf. \cite{HarrisTaylor2001}, \cite{Henniart2000}, \cite{Scholze2013}).
Define $\pi^c$ to be the representation of $\GL_N(E)$ by $\pi^c(a)=\pi(\b{a})$ for $a\in\GL_N(E)$. On the other hand, 
fix $w_0\in WD_F\setminus WD_F$ and define a representation $\phi_\sigma^c$ of $WD_E$ by 
$\phi_\sigma(w_0ww_0^{-1})$ for $w\in WD_E$. The representation $\phi_\sigma^c$ is independent of the choice of 
$w_0$. Now the lemma can be stated as follow:

\begin{lm}\label{L:Galois conj}
Under the local Langlands correspondence for $\GL_N(E)$, we have $\phi^c_{\sigma}\cong\phi_{\sigma^c}$
\end{lm}

\begin{proof}
When $N=1$, the assertion follows from the local class field theory (cf. \cite{Tate1979}). For general $N$, the proof
consists of two steps. The first step is to reduce the proof to the case when $\sigma$ is 
supercuspidal. This part is quite straightforward and it follows from the Bernstein-Zelevinsky classification 
(cf. \cite{Zelevinsky1980}) and the explicit local Langlands correspondence modulo the supercuspidal ones 
(cf. \cite[Section 4.2]{Wedhorn2000}). The second step is obviously to prove the lemma when $\sigma$ is supercuspidal.
In particular, $\sigma$ is generic. The proof is by induction on $N$ together with the local converse theorem for $\GL_N$ 
(see \cite{JacquetLiu2018} and the references therein).\\

Let $\sigma$ be an irreducible supercuspidal representation of $\GL_N(E)$ with $N\ge 2$. Let $\sigma'$ be the irreducible 
representation of $\GL_N(E)$ whose $L$-parameter $\phi_{\sigma'}$ is isomorphic to $\phi_\sigma^c$. Note that since 
$\sigma$ is supercuspidal, $\phi_\sigma$ is irreducible. It follows that $\phi_{\sigma}^c$ is also irreducible and hence 
$\sigma'$ is again supercuspidal. Moreover, $\sigma'$ and $\sigma^c$ have the same central character. Indeed, 
under the local Langlands correspondence, the central character $\omega_\sigma$ corresponds to $\det(\phi_\sigma)$, 
and hence character of $\sigma^c$, which is clearly equal to $\omega^c_\sigma$, corresponds to $\det(\phi_\sigma)^c$.
As the central character of $\sigma'$ corresponds to $\det(\phi_\sigma^c)=\det(\phi_\sigma)^c$, the assertion follows. 
Recall that our goal is to show $\sigma'\cong\sigma^c$. So suppose inductively that the assertion holds for every 
irreducible supercuspidal representations $\tau$ of $\GL_r(E)$ with $1\le r\le N-1$. Thus we have 
\begin{equation}\label{E:con 1}
\phi^c_{\tau}\cong\phi_{\tau^c}
\end{equation}
for every such $\tau$.
To apply the local converse theorem, we have to consider the $\gamma$-factor
$\gamma(s,\sigma\x\tau,\psi)$ attached to $\sigma$, $\tau$ and a non-trivial additive character $\psi$ of $E$ defined by 
the Rankin-Selberg integrals for $\GL_N\x\GL_r$ (cf. \cite{JPSS1983}). Using these local integrals, one can check directly
that 
\begin{equation}\label{E:con 2}
\gamma(s,\sigma^c\x\tau^c,\psi^c)=\gamma(s,\sigma\x\tau,\psi)
\end{equation}
where $\psi^c$ is the character of $E$ defined by $\psi^c(x)=\psi(\b{x})$ for $x\in E$. On the other hand, we also have 
\begin{equation}\label{E:con 3}
\gamma(s,\phi^c_\sigma\ot\phi^c_\tau, \psi^c)=\gamma(s,\phi_\sigma\ot\phi_\tau,\psi).
\end{equation}
To verify this, first notice that 
\[
\gamma(s,\phi^c_\sigma\ot\phi^c_\tau, \psi^c)
=
\gamma(s,\sigma'\x\tau^c, \psi^c)
=
\e(s,\sigma'\x\tau^c,\psi^c)
=
\e(s,\phi_\sigma^c\ot\phi^c_\tau,\psi^c)
\]
as both of the $L$-factors appeared in the $\gamma(s,\sigma'\x\tau^c,\psi^c)$ are equal to $1$ 
(cf. \cite[Proposition 8.1]{JPSS1983}). Here we also use the fact that these local factors are preserved under the local 
Langlands correspondence. Since same reasoning shows 
$\gamma(s,\phi_\sigma\ot\phi_{\tau},\psi)=\e(s,\phi_\sigma\ot\phi_\tau,\psi)$, we are reducing to show 
\[
\e(s,\phi_\sigma^c\ot\phi^c_\tau,\psi^c)
=
\e(s,\phi_\sigma\ot\phi_\tau,\psi).
\]
But this is a consequence of the following identity: 
\begin{equation}\label{E:con 4}
\e(s,\phi^c,\psi^c)=\e(s,\phi,\psi)
\end{equation}
for every admissible representations $\phi$ of $WD_E$, whose validity can be verified by Tate's integral formula 
(cf. \cite{Tate1979}) and the inductivity of the $\e$-factors.\\
 
Now we can complete the proof. In fact, by \eqref{E:con 1}, \eqref{E:con 2} and \eqref{E:con 3}, we have
\begin{align*}
\gamma(s,\sigma^c\x\tau,\psi)
=
\gamma(s,\sigma\x\tau^c,\psi^c)
=
\gamma(s,\phi_\sigma\x\phi_{\tau^c},\psi^c)
=
\gamma(s,\phi_\sigma\x\phi^c_\tau,\psi^c)
=
\gamma(s,\phi_\sigma^c\x\phi_\tau,\psi)
=
\gamma(s,\sigma'\x\tau,\psi)
\end{align*}
for every supercuspidal representations $\tau$ of $\GL_r(E)$ with $1\le r\le N-1$
\footnote{Actually, $\lfloor\frac{N}{2}\rfloor$ is enough.}. 
Since $\sigma^c$ and $\sigma'$ have the same central, we conclude that $\sigma^c\cong\sigma'$ by the local converse 
theorem for $\GL_N$ (cf. \cite{JacquetLiu2018}).
\end{proof}

\lmref{L:Galois conj} has two corollaries. To state the first one, let $\sigma$ be an irreducible representation of 
$\GL_N(E)$ with the associated $L$-parameter $\phi_\sigma$ as before. Since $\psi_E$ unramified, the $\e$-factor 
$\e(s,\phi_\sigma,\psi_E)$ can be written as
\[
\e(1/2,\phi_\sigma,\psi_E)q^{-a_\sigma(2s-1)}
\]
for some integer $a_\sigma\ge 0$ and non-zero complex number $\e(1/2,\phi_\sigma,\psi_E)$. 
Denote by $\tilde{\sigma}$ the contragredient of $\sigma$. Then we have:

\begin{cor}\label{C:epsilon}
Let $\sigma$ be an irreducible representation of ${\rm GL}_N(E)$. Then $a_{\tilde{\sigma}^c}=a_\sigma$. 
\end{cor}

\begin{proof}
Since $\psi_E$ is trivial on $F$, one has $\psi_E^c=\psi_E^{-1}$. In particular, $\psi_E^c$ is again unramified, and 
hence $a_{\tilde{\sigma}^c}$ also appears in the exponent of $\e(s,\phi_{\tilde{\sigma}^c},\psi_E^c)$. 
Since
\[
\e(s,\phi_{\tilde{\sigma}^c},\psi_E^c)=\e(s,\phi^c_{\tilde{\sigma}},\psi_E^c)=\e(s,\phi_{\tilde{\sigma}},\psi_E)
\]
by \lmref{L:Galois conj} and \eqref{E:con 4}, we see that $a_{\tilde{\sigma}^c}=a_{\tilde{\sigma}}$. 
Next, since $\phi_{\tilde{\sigma}}\cong\tilde{\phi}_\sigma$, where
$\tilde{\phi}_\sigma$ is the contragredient representation of $\phi_\sigma$ 
and
\[
\e(s,\phi_\sigma,\psi_E)\e(1-s,\tilde{\phi}_\sigma,\psi_E^{-1})=1
\]
by \cite[Equation (3.4.7)]{Tate1979}, we conclude that $a_{\tilde{\sigma}}=a_\sigma$. This completes the proof.
\end{proof}

\begin{cor}\label{C:L-factor}
Let $\phi:WD_E\to\GL_N(\bbC)$ be an $L$-parameter. If $\t{\phi}\cong\phi^c$, then $L(s,\t{\phi})=L(s,\phi)$.
\end{cor}

\begin{proof}
Let $\sigma$ be an irreducible representation of $\GL_N(E)$ corresponds to $\phi$ under the local Langlands 
correspondence. Then the assumption on $\phi$ and \lmref{L:Galois conj} imply 
\[
L(s,\t{\sigma})=L(s,\t{\phi})=L(s,\phi^c)=L(s,\sigma^c)
\]
where the first and the last $L$-factors are defined by the Rankin-Selberg integrals in \cite{JPSS1983}. Thus it suffices 
to show that $L(s,\sigma^c)=L(s,\sigma)$. But this follows immediately from the definition of the integrals.
\end{proof}

To state the next lemma, let $\tilde{J}_N\in\GL_N(E)$ be the element defined inductively by 
\[
\tilde{J}_1=(1)
\quad\text{and}\quad
\tilde{J}_{N}
=
\begin{pmatrix}
&\tilde{J}_{N-1}\\
(-1)^{N-1}&
\end{pmatrix}
\]
and $\theta: {\rm GL}_{2n+1}(E)\to{\rm GL}_{2n+1}(E)$ be the involution given by 
\begin{equation}\label{E:inv}
a^\theta=\tilde{J}_{2n+1}{}^t\bar{a}^{-1}\tilde{J}_{2n+1}^{-1}.
\end{equation}
Let $\tilde{K}_{2n+1,m}\subset {\rm GL}_{2n+1}(E)$ to be the open compact subgroup defined by 
\begin{equation}\label{E:tilde K}
\tilde{K}_{2n+1,m}=\omega_m\Gamma_{2n+1,m}\omega_m^{-1}
\quad\text{where}\quad
\omega_m
:=
\begin{pmatrix}
I_n&&\\
&&1\\
&\varpi^m I_n
\end{pmatrix}
\in
\GL_{2n+1}(E).
\end{equation}
Then $\tilde{K}_{2n+1,m}$ consists of matrices $k$ of the form
\[
\bordermatrix{
              & n     & 1    & n     \cr
    n     & \mathfrak{o}_E &\mathfrak{o}_E &\mathfrak{p}_E^{-m}     \cr
    1     &\mathfrak{p}^m_E& 1+\mathfrak{p}^m_E &\mathfrak{o}_E\cr
    n     &\mathfrak{p}^m_E&\mathfrak{p}^m_E& \mathfrak{o}_E   \cr
            }
\] 
with $\det(k)\in\frak{o}_E^{\x}$, and is invariant under $\theta$.
Let $\mathit{\Pi}$ be an irreducible generic representation of ${\rm GL}_{2n+1}(E)$ and define $\itPi^\theta$ to be an 
irreducible generic representation of ${\rm GL}_{2n+1}(E)$ acting on $\cV_\sigma$ with the action 
$\itPi^\theta(a)=\itPi(a^\theta)$. We assume that $\itPi\cong\itPi^\theta$. 
Since $(\itPi^\theta)^\theta=\itPi$, there exists an intertwining map 
$I:\itPi\overset{\sim}{\longrightarrow}\itPi^\theta$ with $I^2={\rm id}$. This $I$ is unique up to $\pm 1$ and can be 
normalized in the following way. On one hand, since $\tilde{K}_{2n+1,m}$ is $\theta$-invariant, $I$ preserves the space 
$\cV_\itPi^{\tilde{K}_{2n+1}, m}$ for every $m\ge 0$. On the other, \eqref{E:tilde K} and the theory of local newforms 
for generic representations of ${\rm GL}_N$ (\cite{JPSS1981}, \cite{Jacquet2012}) imply
\[
{\rm dim}_{\mathbb{C}}\cV_\itPi^{\tilde{K}_{2n+1}, a_\itPi}
=
1.
\]
Thus we can require that $I$ is the identity map on the space $\cV_\itPi^{\tilde{K}_{2n+1}, a_\itPi}$. 
In particular, the trace of $I$ on the space $\cV_\itPi^{\tilde{K}_{2n+1}, a_\itPi}$ is $1$. The next
lemma computes the trace of $I$ on the space $\cV_\itPi^{\tilde{K}_{2n+1}, m}$ for every $m\ge 0$.

\begin{lm}\label{L:trace}
Let notation be as above. We have 
\[
{\rm tr}(I; \cV_\itPi^{\tilde{K}_{2n+1},m})
=
\begin{pmatrix}
\lfloor\frac{m-a_\itPi}{2}\rfloor+n\\
n
\end{pmatrix}
\]
for every $m\ge 0$.
\end{lm}

\begin{proof}
Let $I_m:\cV_\itPi\overset{\sim}{\longto}\cV_\itPi$ be the $\mathbb{C}$-linear isomorphism defined by 
$I_m=\itPi(\omega_m^{-1})\circ I\circ\itPi(\omega_m)$. Then from \eqref{E:tilde K} and the fact that 
and $I$ preserves $\cV_\itPi^{\tilde{K}_{2n+1,m}}$, we see that $I_m$ preserves 
$\cV_\itPi^{\Gamma_{2n+1,m}}$ and
\[
{\rm tr}(I;\cV_\itPi^{\tilde{K}_{2n+1},m})
=
{\rm tr}(I_m;\cV_\itPi^{\Gamma_{2n+1,m}}).
\]
So it suffices to compute the trace of $I_m$ on $\cV_\itPi^{\Gamma_{2n+1,m}}$. To do so, we first recall the results 
of Jacquet--Piatetski-Shapiro--Shalika (\cite{JPSS1981}, see also \cite{Jacquet2012}) and Reeder (\cite{Reeder1991}).
Let $\cH=\cH({\rm GL}_{2n}(E)//{\rm GL}_{2n}(\frak{o}_E))$ be the spherical Hecke algebra of ${\rm GL}_{2n}(F)$. 
We embed ${\rm GL}_{2n}(E)$ into ${\rm GL}_{2n+1}(E)$ via $a\mapsto\pMX{a}{}{}{1}$ and define the action of 
$\cH$ on $\cV_\itPi^{{\rm GL}_{2n}(\frak{o}_E)}$ by 
\[
f\star v
=
\int_{{\rm GL}_{2n}(E)}f(a)\itPi\left(\pMX{a^{-1}}{}{}{1}\right)v|\det(a)|_E^{\frac{1}{2}}da
\quad 
\]
for $f\in\cH$ and $v\in\cV_\itPi^{{\rm GL}_{2n}(\frak{o}_E)}$. Here the Haar measure $da$ on $\GL_{2n}(E)$ is 
chosen so that ${\rm vol}({\rm GL}_{2n}(\frak{o}_E), da)=1$. Note that 
$\cV_\itPi^{\Gamma_{2n+1,m}}\subseteq\cV_\itPi^{{\rm GL}_{2n}(\frak{o}_E)}$ for every $m\ge 0$ and we have 
the identity
\begin{equation}\label{E:Hecke action}
(f\star f')\star v=f\star(f'\star v)
\end{equation}
for every $f, f'\in\cH$ and $v\in\cV_\sigma^{{\rm GL}_{2n}(\frak{o}_E)}$. Indeed, a direct computation shows 
$(f\star f')\star v=f'\star(f\star v)$ for every $f, f'\in\cH$ and $v\in\cV_\sigma^{{\rm GL}_{2n}(\frak{o}_E)}$. But since 
$\cH$ is abelian, the identity \eqref{E:Hecke action} follows.\\ 

From a result of \cite{JPSS1981}, we know that $\cV_\itPi^{\Gamma_{2n+1,m}}=0$ for $0\le m<a_\itPi$ and 
$\cV_\itPi^{\Gamma_{2n+1,m}}\neq 0$ for all $m\ge a_\itPi$ with
${\rm dim}_{\mathbb{C}}\cV_\itPi^{\Gamma_{2n+1,a_\itPi}}=1$. Fix a non-zero element 
$v_0\in\cV_\itPi^{\Gamma_{2n+1,a_\itPi}}$. Then our normalization on $I$ implies
\begin{equation}\label{E:normalization}
I_{a_\itPi}(v_0)=v_0.
\end{equation}
To describe a basis of $\cV_\itPi^{\Gamma_{2n+1,m}}$ for $m> a_\itPi$, 
we recall a result of Reeder (\cite{Reeder1991}). For $a\in{\rm GL}_{2n}(E)$, we denote 
\[
[a]
=
{\rm GL}_{2n}(\frak{o}_E)a{\rm GL}_{2n}(\frak{o}_E).
\]
For $0\le i\le 2n$, we write
\[
\ul{\varpi}_i
=
{\rm diag}(\underbrace{\varpi,\ldots,\varpi}_{i},\underbrace{1,\ldots, 1}_{2n-i})\in{\rm GL}_{2n}(E)
\]
and put 
\begin{equation}\label{E:basis}
f_i=q_E^{i(2n-i)/2}\mathbb{I}_{[\ul{\varpi}_i]}.
\end{equation}
Then
\begin{equation}\label{E:basis GL}
\beta_\ell
:=
\stt{f_0^{\ell_0}f_1^{\ell_1}\cdots f_{2n}^{\ell_{2n}}\star v_0\mid\text{$\ell_i\ge 0$ for $0\le i\le 2n$ and
$\ell_0+\ell_1+\cdots+\ell_{2n}=\ell$}}
\end{equation}
is a basis of $\cV_\itPi^{\Gamma_{2n+1,a_\itPi+\ell}}$ for every $\ell\ge 0$.\\

To proceed, we prove the following identity, which is the key to the proof of this lemma. Let $f\in\cH$,
$v\in\cV_\sigma^{\Gamma_{2n+1,m}}$ and suppose that $f\star v\in\cV_\sigma^{\Gamma_{2n+1,m+1}}$. Then 
we have 
\begin{equation}\label{E:key id}
I_{m+1}(f\star v)
=
f^\iota\star I_m(v)
\end{equation}
with $f^\iota\in\cH$ defined by
\[
f^\iota(a)
=
q_E^{-n}f(\varpi{}^t\bar{a}^{-1})|\det(a)|^{-1}
\]
for $a\in{\rm GL}_{2n}(E)$, where we view $\varpi$ as an element in the center of ${\rm GL}_{2n}(E)$.
The proof of this identity is by a straightforward computation:
\begin{align*}
I_{m+1}(f\star v)
&=
\int_{{\rm GL}_{2n}(E)}
f(a) (I_{m+1}\circ\itPi)\left(\pMX{a^{-1}}{}{}{1}\right)v|\det(a)|^{\frac{1}{2}}da\\
&=
\int_{{\rm GL}_{2n}(E)}
f(a)(\itPi(\omega_{m+1}^{-1})\circ I\circ\itPi)\left(\omega_{m+1}\pMX{a^{-1}}{}{}{1}\omega_m^{-1}\right)
\itPi(\omega_m)v|\det(a)|^{\frac{1}{2}}da\\
&=
\int_{{\rm GL}_{2n}(E)}
f(a)\itPi\left(\omega_{m+1}^{-1}\omega^\theta_{m+1}\pMX{a^{-1}}{}{}{1}^\theta(\omega^{-1}_m)^\theta\right)
I(\itPi(\omega_m)v)|\det(a)|^{\frac{1}{2}}da\\
&=
\int_{{\rm GL}_{2n}(E)}
f(a)\itPi\left(\omega_{m+1}^{-1}\omega^\theta_{m+1}\pMX{a^{-1}}{}{}{1}^\theta(\omega^{-1}_m)^\theta\omega_m\right)
I_m(v)|\det(a)|^{\frac{1}{2}}da\\
&=
\int_{{\rm GL}_{2n}(E)}
f(a)\itPi\left(
\pMX{\varpi^{-1}J{}^t\bar{a}J^{-1}}{}{}{1}\right)I_m(v)|\det(a)|^{\frac{1}{2}}da\\
&=
\int_{{\rm GL}_{2n}(E)}
f^\iota(a)\itPi\left(
\pMX{a^{-1}}{}{}{1}\right)I_m(v)
|\det(a)|^{\frac{1}{2}}db\\
&=
f^\iota\star I_m(v)
\end{align*}
where $J:=\pMX{I_n}{}{}{-I_n}\tilde{J}_{2n}\in{\rm GL}_{2n}(\frak{o}_E)$. This proves \eqref{E:key id}.\\

Recall that $f_i\in\cH$ is given by \eqref{E:basis} for $0\le i\le 2n$. One can check directly that 
\begin{equation}\label{E:I on f_i}
f^\iota_i
=
q_E^{n-i}f_{2n-i}.
\end{equation}
Now we are ready to compute ${\rm tr}(I_m;\cV_\itPi^{\Gamma_{2n+1,m}})$. We may assume $m\ge a_\itPi$
since otherwise the space $\cV_\itPi^{\Gamma_{2n+1,m}}$ is zero and so is the trace. Let $m=a_\itPi+\ell$ for 
some $\ell\ge 0$. To obtain the trace, we compute the matrix of $I_m$ with respect to the basis \eqref{E:basis GL}.
Let $v=f_0^{\ell_0}f_1^{\ell_1}\cdots f_{2n}^{\ell_{2n}}\star v_0\in\beta_\ell$ so that $\ell_0+\ell_1+\cdots+\ell_{2n}=\ell$. 
Then we have 
\[
I_m(v)
=
I_{a_\itPi+\ell}(v)
=
q_E^{\sum_{i=0}^{2n}\ell_i(n-i)}f_0^{\ell_{2n}}f_1^{\ell_{2n-1}}\cdots f_{2n}^{\ell_0}\star v_0
\]
by \eqref{E:Hecke action}, \eqref{E:normalization}, \eqref{E:key id} and \eqref{E:I on f_i}. From this we see that 
$I_{a_\itPi+\ell}(v)$ is constant multiple of an element in $\beta_\ell$, and moreover,  
\[
I_{a_\itPi+\ell}(v)\in\mathbb{C}v 
\]
if and only if $\ell_i=\ell_{2n-i}$ for $0\le i\le 2n$, in which case we have $I_{a_\itPi+\ell}(v)=v$.
It follows that 
\begin{align*}
{\rm tr}\left(I_{a_\itPi+\ell}; \cV_{\itPi}^{\Gamma_{2n+1,a_\itPi+\ell}}\right)
&=
\left|\stt{(\ell_0,\ldots,\ell_{2n})\in\mathbb{Z}_{\ge 0}^{2n+1}\mid\text{$\ell_i=\ell_{2n-i}$ for $0\le i\le 2n$ and 
$\ell_0+\cdots +\ell_{2n}=\ell$}}\right|\\
&=
\left|\stt{(\ell_0,\ldots,\ell_{n})\in\mathbb{Z}_{\ge 0}^{n+1}\mid 2\ell_0+\cdots +2\ell_{n-1}+\ell_n=\ell}\right|\\
&=
\left|\stt{(\ell_0,\ldots,\ell_{n-1})\in\mathbb{Z}_{\ge 0}^{n}\mid \ell_0+\cdots +\ell_{n-1}\le\lfloor\ell/2\rfloor}\right|\\
&=
\begin{pmatrix}
\lfloor\frac{\ell}{2}\rfloor+n\\
n
\end{pmatrix}.
\end{align*}
This completes the proof.
\end{proof}

\subsection{Proof of \thmref{T:main}}
In the following proof, we will retain the notation in the previous subsection. 
Let $\pi$ be an irreducible generic representation of $G_n$ with the associated $L$-parameter 
$\phi_\pi: WD_E\to{\rm GL}_{2n+1}(\mathbb{C})$ which is conjugate orthogonal 
(cf. \cite[Section 8]{GanGrossPrasad2012}, \cite{Mok2015}). 
Then by the local Langlands correspondence for $\GL_{2n+1}$, $\phi_\pi$ 
corresponds to an irreducible representation $\itPi$ of $\GL_{2n+1}(E)$. This $\itPi$ is conjugate self-dual and 
therefore $\itPi\cong\itPi^\theta$. Assume first that $\pi$ is tempered. Then $\itPi$ is also tempered and hence 
generic. Let $I:\itPi\overset{\sim}{\longto}\itPi^\theta$ be the normalized intertwining map as in the paragraph before 
\lmref{L:trace}. Then the proofs of \cite[Theorems 4.3, 4.4]{AtobeOiYasuda} imply
\[
{\rm dim}_{\mathbb{C}}\cV_\pi^{K_{n,m}}={\rm tr}(I; \cV_\itPi^{\tilde{K}_{2n+1,m}})
\]
for every $m\ge 0$. Now the desired identity for tempered $\pi$ follows from \lmref{L:trace}.\\

Suppose that $\pi$ is non-tempered from now on. Then by the Langlands' classification 
(\cite{Silberger1978}) and standard module conjecture (\cite{CasselmanShahidi1998}, \cite{Muic2001}), 
\begin{equation}\label{E:pi}
\pi
\cong
\tau_1\x\cdots\x\tau_k\rtimes\pi_0
:=
{\rm Ind}_{P_{n,\ul{r}}}^{G_n}(\tau_1\boxtimes\cdots\boxtimes\tau_k\boxtimes\pi_0)
\quad
(\text{normalized induced})
\end{equation}
where $1\le r\le n$ is an integer, $\ul{r}=(r_1,\ldots, r_k)$ is a partition of $r$, $P_{n,\underline{r}}\subset G_n$ is the 
parabolic subgroup containing $B_n$ whose Levi subgroup isomorphic to 
${\rm GL}_{r_1}\x\cdots\x{\rm GL}_{r_k}\x G_{n-r}$, $\tau_j$ is an irreducible essentially square integrable 
representations of ${\rm GL}_{r_j}$ for $1\le j\le k$, and $\pi_0$ is an irreducible tempered generic representation of 
$G_{n-r}$. The associated $L$-parameter then decompose accordingly 
\[
\phi_\pi
=
\phi_{\tau_1}\oplus\cdots\oplus\phi_{\tau_k}
\oplus
\phi_{\pi_0}
\oplus
\tilde{\phi}^c_{\tau_k}\oplus\cdots\oplus\tilde{\phi}^c_{\tau_1}
\]
where $\phi_{\tau_j}$ (resp. $\phi_{\pi_0}$) is the $L$-parameter attached to $\tau_j$ (resp. $\pi_0$) for 
$1\le j\le k$. Since 
\[
\e(s,\phi_\pi,\psi_E)
=
\e(s,\phi_{\pi_0},\psi_E)
\prod_{j=1}^k\e(s,\phi_{\tau_j},\psi_E)\e(s,\tilde{\phi}^c_{\tau_j},\psi_E)
\]
we find that 
\[
a_\pi
=
a_{\pi_0}+2\sum_{j=1}^k a_{\tau_j}
\]
by \corref{C:epsilon}.\\

Since $K_{n,m}$ is conjugate to $K^0_{n,m}$ for each $m$, it suffices to prove \eqref{E:dim formula} with 
$\cV_\pi^{K_{n,m}}$ replaced by $\cV_\pi^{K^0_{n,m}}$. We do this by the induction on $k$. Suppose that $k=1$ so that 
$r=r_1$, $\pi=\tau_1\rtimes\pi_0$ and  $a_\pi=a_{\pi_0}+2a_{\tau_1}$.
Write $m=e+2\ell$ for some $e=0,1$ and $\ell\ge 0$. Then \corref{C:inter Levi} implies
\[
{\rm dim}_{\mathbb{C}}\cV_{\pi}^{K^0_{n,m}}
=
\sum_{d=0}^{\ell}
{\rm dim}_{\mathbb{C}}\left(
\cV_{\tau_1}\otimes\cV_{\pi_0}
\right)^{M^d_{n,r}}
=
\sum_{d=0}^{\ell}
\left(
{\rm dim}_{\mathbb{C}}
\cV_{\tau_1}^{\Gamma_{r,\ell-d}}
\right)
\left(
{\rm dim}_{\mathbb{C}}
\cV_{\pi_0}^{K^0_{n-r,e+2d}}
\right).
\]
Observe that if $\ell-d<a_{\tau_1}$ or $e+2d<a_{\pi_0}$, then $\cV_{\tau_1}^{\Gamma_{r,d-\ell}}=0$ or 
$\cV_{\pi_0}^{K^0_{n-r,e+2d}}=0$ accordingly by \thmref{T:main'} and the theory of local newforms for ${\rm GL}_r$ 
(cf. \cite{JPSS1981}). It follows that $\cV_\pi^{K^0_{n,m}}=0$ if $m<a_\pi$. Now suppose that $m\ge a_\pi$. 
If $m$ and $a_\pi$ have the same parity, then we can write $m=a_\pi+2\ell_1=a_{\pi_0}+2(a_{\tau_1}+\ell_1)$ for some 
$\ell_1\ge 0$. Note that $m$ and $a_{\pi_0}$ also have the same parity. Then the above observation and the dimension 
formulae for $\pi_0$ (cf. \eqref{E:dim formula}) and $\tau_1$ (cf. \cite{Reeder1991}) give
\begin{align*}
{\rm dim}_{\mathbb{C}}\cV_{\pi}^{K^0_{n,m}}
&=
\sum_{d=0}^{\ell}
\left(
{\rm dim}_{\mathbb{C}}
\cV_{\tau_1}^{\Gamma_{r,\ell-d}}
\right)
\left(
{\rm dim}_{\mathbb{C}}
\cV_{\pi_0}^{K^0_{n-r,e+2d}}
\right)\\
&=
\sum_{d=0}^{\ell_1}
\left(
{\rm dim}_{\mathbb{C}}
\cV_{\tau_1}^{\Gamma_{r,a_{\tau_1}+\ell_1-d}}
\right)
\left(
{\rm dim}_{\mathbb{C}}
\cV_{\pi_0}^{K^0_{n-r,a_{\pi_0}+2d}}
\right)\\
&=
\sum_{d=0}^{\ell_1}
\begin{pmatrix}
r-1+\ell_1-d\\
\ell-d
\end{pmatrix}
\begin{pmatrix}
\ell_1+n-r\\
n-r
\end{pmatrix}\\
&=
\begin{pmatrix}
\ell_1+n\\
n
\end{pmatrix}.
\end{align*}
The last equality follows from the combinatorial identity in \cite[(3.2)]{Gould1972}. This proves \eqref{E:dim formula} when 
$m$ and $a_\pi$ have the same parity. The proof when $m$ and $a_\pi$ have the opposite parity is similar. We just need 
to replaced $a_{\pi}$ by $a_{\pi}+1$. The proof for $k=1$ is now complete. Suppose inductively that 
\eqref{E:dim formula} holds for $k-1$ and $\pi$ is of the form \eqref{E:pi}. Then by induction in stage, we can write 
$\pi\cong\tau_1\rtimes \pi_1$ where $\pi_1:=\tau_2\x\cdots\x\tau_k\rtimes\pi_0$ is an irreducible generic 
representation of $G_{n-r_1}$. Since $a_\pi=a_{\pi_1}+2a_{\tau_1}$ and \eqref{E:dim formula} holds for $\pi_1$ by 
the induction hypothesis, we can apply the argument for $k=1$ to obtain \eqref{E:dim formula} for $\pi$. 
This finishes the proof.\qed

\section{Conjectural bases for oldforms}\label{S:basis}
Now we come to the second part of this paper, namely, we would like to compute the Rankin-Selberg integrals attached 
newforms and also oldforms. As we will see, many statements and their proofs in this part of the paper are similar to that 
of \cite{YCheng2022}. So in these cases, we will simply write down the statements and refer their proofs to those in 
loc. cit.. However, we will indicate the modifications whenever needed.

\subsection{Level raising operators}
As in the literature, the conjectural bases for oldforms are obtained from certain level raising procedures.
To define the level raising operators, we begin with the following lemma whose proof is similar to that of 
\cite[Lemma 3.1]{YCheng2022}.

\begin{lm}\label{L:same vol}
Let $dh$ be a Haar measure on $H_r$. Then we have ${\rm vol}(R_{r,m}, dh)={\rm vol}(R_{r,0},dh)$ for all $m\geq 0$.
\end{lm}

\begin{proof}
Let $t_{r,m}=\diag{\varpi^mI_r, I_r}$. Then $t_{r,m}\in{\rm GU}_{2r}(F)$, the similitude unitary group with $2r$ variables 
and $t_{r,m}R_{r,m}t_{r,m}^{-1}=R_{r,0}$. Now we can apply the proof of \cite[Lemma 3.1]{YCheng2022} to obtain the 
lemma. Note that the intersection of the center of ${\rm GU}_{2r}(F)$ and $\U_{2r}(F)$ is $E^1$, which is compact. So the
proof of loc. cit. is still applicable.
\end{proof}

Following \cite{Reeder1991} and \cite{Tsai2013}, \cite{YCheng2022}, the level raising operators coming from the elements 
in the Hecke algebras $\sH(H_n//R_{n,m})$ for $m\ge 0$. To describe a $\mathbb{C}$-linear basis of 
$\sH(H_n//R_{n,m})$, let 
\[
\sP
=
\stt{\lambda=(\lambda_1,\ldots,\lambda_n)\in\mathbb{Z}^n\mid \lambda_1\ge \cdots\ge \lambda_n\ge 0}.
\]
Given $\la\in\sP$, we denote
\[
\varpi^\la
=
\diag{\varpi^{\la_1},\ldots,\varpi^{\la_n},\varpi^{-\la_n},\ldots,\varpi^{-\la_1}}\in T_n
\]
and put
\[
\varphi_{\la,m}
=
\mathbb{I}_{R_{n,m}\varpi^\la R_{n,m}}\in\sH(H_n//R_{n,m}).
\]
Then we have
\[
\sH(H_n(F)//R_{n,m})=\bigoplus_{\lambda\in\sP}\bbC\cdot\varphi_{\lambda, m}.
\]
as $\mathbb{C}$-linear spaces.\\

Now let $\pi$ be an irreducible generic representation of $G_n$ and $v\in\cV_\pi^{R_{n,m}}$. We define 
\[
\varphi\star v
=
\int_{H_n}\varphi(h)\pi(h^{-1})vdh
\quad
\text{(${\rm vol}(R_{n,0},dh)=1$).}
\]
where $\varphi\in\sH(H_n//R_{n,m})$. Clearly, we have $\varphi\star v\in\cV_\pi^{R_{n,m}}$. In particular, if 
$v\in\cV_\pi^{K^0_{n,m}}$ and $\varphi\in\sH(H_n//R_{n,e})$, where $e=0,1$ is such that $m\equiv e\pmod 2$, then 
$\varphi\star v\in\cV_\pi^{K^0_{n,m'}}$ for some $m'\equiv e\pmod 2$ by the filtration \eqref{E:filtration}. 
The following lemma tells us what $m'$ is.

\begin{lm}\label{L:level raising for K^0}
Let $v\in\cV_\pi^{K^0_{n,m}}$ and $\la=(\la_1,\ldots,\la_n)\in\sP$. Then we have 
$\varphi_{\la,e}\star v\in\cV_\pi^{K^0_{n,m+2\la_1}}$.
\end{lm}

\begin{proof}
The proof of this lemma is similar to that of \cite[Proposition 8.1.1]{Tsai2013}. However, since Tsai's proof is very brief, we 
shall fill in some details here. Let us write $m=e+2\ell$ for some $\ell\ge 0$.  Let $v_1=\pi({\varpi^{-\la}})v$ and 
\[
v_2=\int_{R_{n,e}}\pi(h)v_1dh.
\]
Then $v_1$ is $\varpi^{-\la}K^0_{n,m}\varpi^{\la}$-invariant and $v_2$ is a non-zero multiple of $\varphi_{\la,e}\star v$.
Thus it suffices to show that $v_2$ is fixed by $K^0_{n,m+2\la_1}$. For this, we use the decomposition 
\eqref{E:decomp K^0} of $K^0_{n,m}$ when $m\ge 1$ and the fact that $K^0_{n,0}=K_{n,0}$ is generated by $E^1$, 
$R_{n,0}$ and the root groups $\chi_{\pm\e_i}(\frak{o}_E)$ for $1\le i\le n$.\\

Certainly, $v_2$ is fixed by $R_{n,e}$ and 
$E^1_{m+2\la_1}$. Therefore, we only need to check that $v_2$ is invariant under $\chi_{\e_i}(\frak{p}_E^{\ell+\la_1})$ 
and $\chi_{-\e_j}(\frak{p}_E^{e+\ell+\la_1})$ for $1\le i,j\le n$ by the facts just mentioned. 
By conjugating these root groups by Weyl elements in 
$\sW_e$ (cf. \eqref{E:Weyl group}), we actually only need to check that $v_2$ is 
$\chi_{\e_n}(\frak{p}_E^{\ell+\la_1})$-invariant. To do so, we claim that $h^{-1}uh\in\varpi^{-\la}K^0_{n,m}\varpi^\la$ for 
all $h\in R_{n,e}$ and $u\in\chi_{\e_n}(\frak{p}_E^{\ell+\la_1})$. Assuming the claim, we find that
\[
\pi(u)v_2
=
\int_{R_{n,e}}\pi(uh)v_1dh
=
\int_{R_{n,e}}\pi(h\cdot h^{-1}uh)v_1dh
=
\int_{R_{n,e}}\pi(h)v_1dh
=
v_2
\]
and the proof follows.\\

To verify the claim, note that $\varpi^{-\la}K^0_{n,m}\varpi^\la$ contains $\chi_{\e_i}(\frak{p}_E^{\ell-\la_i})$, 
$\chi_{-\e_j}(\frak{p}_E^{e+\ell+\la_j})$ for $1\le i,j\le n$ and $\varpi^{-\la}R_{n,e}\varpi^{\la}$. Moreover, $R_{n,e}$
is generated by $T_n\cap R_{n,e}$, $\chi_{\e_i-e_j}(\frak{o}_E)$, $\chi_{-\e_i+\e_j}(\frak{o}_E)$, 
$\chi_{\e_i+\e_j}(\frak{p}_E^{-e})$, $\chi_{-\e_i-\e_j}(\frak{p}_E^e)$ for $1\le i<j\le n$ and also $\chi_{2e_k}(\frak{p}_F^{-e})$
$\chi_{-2e_k}(\frak{p}_F^e)$ for $1\le k\le n$. Here for $y\in F$ and $1\le k\le n$, we put
\[
\chi_{2e_k}(x)=I_{2n+1}+x\delta E_{k, 2n+2-k}
\quad\text{and}\quad
\chi_{-2e_k}(x)=I_{2n+1}+x\delta E_{2n+2-k, k}
\]
where $E_{ij}$ denotes the $(2n+1)$-by-$(2n+1)$ matrix whose $(i,j)$-entry is $1$ and all other entries are $0$, and 
\[
\chi_{\pm 2\e_k}(S)=\stt{\chi_{\pm 2\e_k}(x)\mid x\in S}
\]
for any subset $S$ of $F$. Now let $y\in \frak{p}_E^{\ell+\la_1}$ and put $u=\chi_{e_n}(y)$. To show that 
$h^{-1}uh\in\varpi^{-\la}K^0_{n,m}\varpi^\la$ for every $h\in R_{n,e}$, we may assume that $h$ is an element in one of 
the subgroups generating $R_{n,e}$. Then the claim is proved via case by case verification. For example, one checks 
taht 
\[
\chi_{\e_i-\e_n}(-z)\chi_{\e_n}(y)\chi_{\e_i-\e_n}(z)
=
\chi_{\e_n}(y)\chi_{\e_i+\e_n}(2^{-1}y\b{y}z)\chi_{\e_i}(-yz)
\]
for $1\le i\le n-1$. In particular, if $z\in\frak{o}_E$, then the RHS of the above identity is contained in 
$\varpi^{-\la}K^0_{n,m}\varpi^\la$. Indeed, since $y\in\frak{p}_E^{\ell+\la_1}\subseteq\frak{p}_E^{\ell-\la_1}$ and 
$-yz\in\frak{p}_E^{\ell+\la_1}\subseteq\frak{p}_E^{\ell+\la_i}$, we see that both $\chi_{\e_n}(y)$ and $\chi_{\e_i}(-yz)$
are contained in $\varpi^{-\la}K^0_{n,m}\varpi^\la$. On the other hand, since 
$2^{-1}y\b{y}z\in\frak{p}_E^{2\ell+2\la_1}\subseteq\frak{p}_E^{\la_i+\la_n-e}$, we find that 
$\chi_{\e_i+\e_n}(2^{-1}y\b{y}z)\varpi^{-\la}R_{n,e}\varpi^\la$. This completes the proof.
\end{proof}

Now we can define the level raising operators used in the constrictions of the conjectural bases. Put
\[
\mu_\ell=(\ell,\ldots,\ell)\in\sP
\] 
for $\ell\ge 0$. Let $v\in\cV_\pi^{K_{n,m}}$ and 
write $m=e+2\ell$ for some $e=0,1$ and $\ell\ge 0$. Then $\pi(\varpi^{\mu_\ell})v\in\cV_\pi^{K^0_{n,m}}$; hence
\[
\varphi_{\la,e}\star \pi(\varpi^{\mu_\ell})v\in\cV_{\pi}^{K^0_{n,m+2\la_1}}
\] 
by \lmref{L:level raising for K^0}, where
$\la=(\la_1,\ldots,\la_n)\in\sP$. Since $\cV_\pi^{K^0_{n,m+2\la_1}}\subseteq\cV_\pi^{K^0_{n,m'}}$ if $m'$ and 
$m$ have the same parity and $m'\ge m+2\la_1$, we thus obtain a level raising operator $\eta_{\la,m,m'}$ from 
$\cV_\pi^{K_{n,m}}$ to $\cV_\pi^{K_{n,m'}}$ given by
\[
\eta_{\la,m,m'}(v)
=
\pi(\varpi^{-\mu_{\ell'}})\varphi_{\la,e}\star\pi(\varpi^{\mu_\ell})v
\] 
where $m'=e+2\ell'$ for some $\ell'\ge 0$ and we understand that $\varpi^{-\mu_{\ell'}}=\left(\varpi^{\mu_{\ell'}}\right)^{-1}$.
Note that $\eta_{\mu_0,m,m}$ is the identity map by \lmref{L:same vol} and $\eta_{\la,m,m'}$ raises level by an even 
integer. To raise levels to those with opposite parity, it natural to consider the operators from $\cV_\pi^{K_{n,m}}$ to 
$\cV_\pi^{K_{n,m+1}}$ defined by 
\[
v\mapsto\vol(K_{n,m}\cap K_{n,m+1},dg)^{-1}\int_{K_{n,m+1}}\pi(g)vdg.
\]    
However, we don't consider these operators here. The reason will be clear in the next subsection.

\subsection{Conjectural bases}\label{SS:basis}
Let $\pi$ be an irreducible generic representation of $G_n$ and $v_\pi\in\cV_\pi^{K_{n,a_\pi}}$ be a newform. 
By \thmref{T:main}, the space $\cV_\pi^{K_{n,a_\pi+1}}$ is also one-dimensional. So we simply let $v_\pi'$ be its basis. 
Now let $m\ge a_\pi$ be an integer. If $m$ and $a_\pi$ have same parity, then we put
\[
\sB_{\pi,m}
=
\stt{\eta_{\la,a_\pi,m}(v_\pi)\mid\text{$\la=(\la_1,\ldots,\la_n)\in\sP$ with $2\la_1\le m-a_\pi$}}.
\]
On the other hand, if $m$ and $a_\pi$ have opposite parity, then we set
\[
\sB_{\pi,m}
=
\stt{\eta_{\la,a_\pi+1,m}(v'_\pi)\mid\text{$\la=(\la_1,\ldots,\la_n)\in\sP$ with $2\la_1\le m-a_\pi-1$}}.
\]
We conjecture that $\sB_{\pi,m}$ is a basis of $\cV_\pi^{K_{n,m}}$ for each $m\ge a_\pi$. In fact, it's not hard to 
see that the set $\sB_{\pi,m}$ has right cardinality, namely, the cardinality of $\sB_{\pi,m}$ is equal to the dimension of 
$\cV_\pi^{K_{n,m}}$ given by \eqref{E:dim formula}. We introduce these conjectural bases in order to compute the 
Rankin-Selberg integrals attached to oldforms. Logically, one has to first verify that these $\sB_{\pi,m}$ are linearly 
independent and then compute the integrals attached to the elements in $\sB_{\pi,m}$. Here we directly compute the 
integrals (under the Assumption \ref{H}) without knowing the linear independence. As a consequence of our computations, 
we show that $\cB_{\pi,m}$ are linearly independent when $\pi$ is tempered. This also gives another reason why we 
don't write $v_\pi'$ as a level raising of $v_\pi$, since we can compute the Rankin-Selberg integral attached to $v_\pi'$ 
without appealing to level raising operators, but only use the fact that $\cV_\pi^{K_{n,a_\pi+1}}$ is one-dimensional.
Of course, one can still ask whether or not
\[
\int_{K_{n,a_\pi+1}}\pi(g)v_\pi dg\ne 0.
\]
But we don't investigate this problem due to our purpose.

\section{Rankin-Selberg integrals}\label{S:RS int}
In this section, we introduce the local Rankin-Selberg integrals attached to generic representations 
of $\U_{2n+1}\x{\rm Res}_{E/F}\GL_r$ with $1\le r\le n$ constructed in \cite{Tamir1991} (for $r=n$) and 
\cite{Ben-ArtziSoudry2009} (for $r\le n$). In this section only, we take $\psi_0$ to be an arbitrary non-trivial additive 
character of $F$, $\theta\in E^\x$ with $\b{\theta}=-\theta$ and let $\psi$ be the additive character of $E$ defined by 
$\psi(x)=\psi(\frac{x-\b{x}}{2\theta})$, so that $\psi$ is trivial on $F$.
To describe these integrals, let $Z_r\subset\GL_r(F)$ be the upper triangular 
maximal unipotent subgroup and $\psi_{Z_r}$ be a non-degenerate character of $Z_r$ defined by 
\[
\b{\psi}_{Z_r}(z)
=
\ol{\psi(z_{12}+z_{23}+\cdots+z_{r-1, r})}
\]
for $z=(z_{ij})\in Z_r$. Let $\tau$ be a representation of $\GL_r(F)$, which is of $Whittaker$ $type$ in the sense that
\[
\dim_{\mathbb{C}}{\rm Hom}_{Z_r}(\tau,\psi_{Z_r})=1.
\]
In particular, irreducible generic representations of $\GL_r(E)$ are of Whittaker type.
We fix a non-zero element $\Lambda_{\tau,\b{\psi}}$ in ${\rm Hom}_{Z_r}(\tau,\b{\psi}_{Z_r})$. 
Define $\tau^*$ to be the representation of $\GL_r(F)$ on $\cV_\tau$ with the action $\tau^*(a)=\tau(a^*)$.
Notice that $\tau^*$ is also of Whittaker type and we fix 
$\Lambda_{\tau^*,\b{\psi}}\in {\rm Hom}_{Z_r}(\tau^*,\b{\psi}_{Z_r})$ with
$\Lambda_{\tau^*,\b{\psi}}=\Lambda_{\tau,\b{\psi}}\circ\tau(d_r)$ where 
\[
d_r
=
\begin{pmatrix}
1&&&\\
&-1&&\\
&&\ddots\\
&&&(-1)^{r-1}
\end{pmatrix}
\in\GL_r(E).
\]
Let $s$ be a complex number and $\tau_s$ be a representation of $\GL_r(E)$ on $\cV_\tau$ with the action 
$\tau_s(a)=\tau(a)|\det(a)|_E^{s-\frac{1}{2}}$. 

\subsection{Induced representations}\label{SS:ind rep}
Let $V_r\subset H_r$ be the upper triangular maximal unipotent subgroup and $Q_r\subset H_r$ be the parabolic 
subgroup containing $V_r$ with the Levi decomposition $Q_r=M_r\ltimes Y_r$ where 
\[
M_r
=
\stt{\pMX{a}{}{}{a^*}\mid a\in\GL_r(E)}\cong\GL_r(E)
\]
and $Y_r\subset V_r$. 
By pulling back the homomorphism $Q_r\twoheadrightarrow Q_r/Y_r\cong\GL_r(E)$, $\tau_s$ becomes a 
representation of $Q_r$ on $\cV_\tau$. We then form a normalized induced representation
\[
\rho_{\tau, s}={\rm Ind}_{Q_r}^{H_r}\tau_s
\]
of $H_r$. The representation space $I_r(\tau,s)$ of $\rho_{\tau,s}$ consists of smooth functions 
$\xi_s: H_r\to\cV_\tau$ satisfying 
\[
\xi_s(mnh)=\delta_{Q_r}^{\frac{1}{2}}(m)\tau_s(m)\xi_s(h)
\]
for $m\in M_r$, $n\in Y_r$ and $h\in H_r$, and the action of $H_r$ on $I_r(\tau,s)$ is by the right translation.

\subsection{Intertwining maps}
There is an intertwining map $M(\tau,s)$ from $I_r(\tau,s)$ to $I_r(\tau^*,1-s)$ given by the integral
\[
M(\tau,s)\xi_s(h)
=
\int_{Y_r}\xi_s(w_r^{-1}uh)du
\]
for $\Re(s)\gg 0$ and by meromorphic continuation in general.
Here 
\[
w_r
=
\begin{pmatrix}&I_r\\I_r&\end{pmatrix}
\in 
H_r
\]
and the Haar measure $du$ is defined as a product measure, where each root group is isomorphic to $E$ or $F$, 
and we take the self-dual measure on them with respect to $\psi$ or $\psi_0$. We normalize $M(\tau,s)$ by 
Shahidi's local coefficients $c_\theta(2s-1, \tau, {\rm As}, \psi_0)$. To define them, let us put
\[
f_{\xi_s}(h)=\Lambda_{\tau,\b{\psi}}(\xi_s)(h).
\]
Similarly, for a given $\xi_s^*\in I_r(\tau^*,s)$, let $f_{\xi^*_s}(h)=\Lambda_{\tau^*,\b{\psi}}(\xi^*_s(h))$.
Then the local coefficients are given by
\[
\int_{V_r}f_{\xi_s}(w_ruh)\psi(u_{r,r+1})du
=
c_\theta(2s-1,\tau,{\rm As},\psi_0)\int_{V_r}f_{M(\tau,s)\xi_s}(w_ruh)\psi(u_{r,r+1})du.
\]
where $u=(u_{ij})\in V_r$ and the Haar measure $du$ on $V_r$ is arbitrary. The normalized intertwining maps
$M^\dagger_{\psi_0,\theta}(\tau,s)$ are then defined by
\[
M^\dagger_{\psi_0,\theta}(\tau,s)
=
c_\theta(2s-1,\tau,{\rm As},\psi_0)M(\tau,s).
\]

\subsection{Rankin-Selberg integrals}
Recall that $N_n\subset G_n$ is the upper triangular maximal unipotent subgroup. Define a non-degenerated character 
$\psi_{N_n}$ of $N_n$ by 
\[
\psi_{N_n}(u)=\psi(u_{12}+u_{23}+\cdots+u_{n-1, n}+u_{n,n+1})
\]
for $u=(u_{ij})\in N_n$. Let $\pi$ be a representation of $G_n$ that is of Whittaker type, i.e.
\[
\dim_{\mathbb{C}}{\rm Hom}_{N_n}(\pi,\psi_{N_n})=1.
\]
In particular, irreducible generic representations of $G_n$ are of Whittaker type.
We fix a non-zero element $\Lambda_{\pi,\psi}$ in ${\rm Hom}_{N_n}(\pi,\psi_{N_n})$. 
For $v\in\cV_\pi$, let $W_v$ be the associated 
Whittaker function with respect to $\psi$, that is, $W_v$ is a $\mathbb{C}$-valued function on $G_n$ given by 
$W_v(g)=\Lambda_{\pi,\psi}(\pi(g)v)$.\\

For a given $x\in{\rm Mat}_{(n-r)\x r}(E)$, we denote 
\begin{equation}\label{E:X_n,r}
\tilde{x}=\begin{pmatrix}I_r&&&&\\x&I_{n-r}&&\\&&1\\&&&I_{n-r}&\\&&&-J_r{}^t\b{x} J_{n-r}&I_r\end{pmatrix}\in G_n.
\end{equation}
Then the Rankin-Selberg integral $\Psi_{n,r}(v\ot\xi_s)$ attached to $v\in\cV_\pi$ and 
$\xi_s\in I_r(\tau,s)$ is given by
\begin{equation}\label{E:RS int}
\Psi_{n,r}(v\ot\xi_s)
=
\int_{V_r\backslash H_r}\int_{{\rm Mat}_{(n-r)\x r}(E)}
W_v(\tilde{x}h)f_{\xi_s}(h)dxdh.
\end{equation}
The Haar measures $dx$ and $dh$ can be any at this moment and we embed $H_r$ into $G_n$ via the embedding 
given in \S\ref{SSS:embed}. As usual, these integrals, which are originally absolute convergence in some half planes, 
have meromorphic continuations to whole complex plane, and give rise to rational functions in $q_E^{-s}$. 
Moreover, the following functional equations
\begin{equation}\label{E:FE}
\Psi_{n,r}(v\ot M^\dagger_{\psi_0,\theta}(\tau,s)\xi_s)=\gamma_\theta(s,\pi\x\tau,\psi_0)\Psi_{n,r}(v\ot\xi_s)
\end{equation}
are satisfied for every $v\in\cV_\pi$ and $\xi_s\in I_r(\tau,s)$, where $\gamma_\theta(s,\pi\x\tau,\psi_0)$ is a nonzero 
rational function in $q^{-s}_E$ depending only on $\psi_0$, $\theta$ and (the classes of) $\pi$, $\tau$.

\begin{remark}
We have some comments about these integrals. First, our definition of $\Psi_{n,r}(v\ot\xi_s)$ is slightly different from 
(but equivalent to) that of \cite{Ben-ArtziSoudry2009} because we want to emphasis the resemblance between these 
integrals and the ones attached to generic representations of special orthogonal groups (cf. \cite{Soudry1993}). Second,
in \cite{Ben-ArtziSoudry2009}, the integrals are attached to irreducible generic representations and the results about 
them, namely, admit meromorphic continuations, are rational functions in $q_E^{-s}$ and can be made into a non-zero
constant, are proved for irreducible ones. However, their proofs are easily apply to our more general settings. The only
less obvious statement is the existence of the functional equations. In fact, in loc. cit., the authors did not investigate
the functional equations of these integrals. On the other hand, Q. Zhang proved the existence of the functional equations 
for irreducible representations in \cite{QZhang2019} by applying the uniqueness of Bessel model for $\U_{2n+1}$ 
established in \cite{AGRS2010} and \cite{GanGrossPrasad2012}. In the Appendix, we will extend Zhang's result to 
representations of Whittaker type.
\end{remark}

The following lemma explains why we only consider unramified representations of $\GL_r(E)$ when computing the 
Rankin-Selberg integrals.

\begin{lm}\label{L:vanish}
Suppose that $v\in\cV_\pi^{K_{n,m}}$ is a nonzero element. Then the integrals $\Psi_{n,r}(v\ot\xi_s)$ vanish 
for all $\xi_s\in I_r(\tau,s)$ if $\tau$ is not unramified. Furthermore, if $r=n$, $\tau$ is unramified and $W_v$ is 
identically zero on $T_n$, then the integrals $\Psi_{n,n}(v\ot\xi_s)$ also vanish for all $\xi_s\in I_n(\tau,s)$. 
\end{lm}  

\begin{proof}
The proof of this lemma is similar to that of \cite[Lemma 4.3]{YCheng2022}. Indeed, the decomposition 
$H_r=Q_rR_{r,m}$ and the fact $M_r\cap R_{r,m}$ give the isomorphism 
$
I_r(\tau,s)^{R_{r,m}}\cong\cV_\tau^{\GL_r(\frak{o}_E)}
$
between $\mathbb{C}$-linear spaces. Now if $\xi_s\in I_r(\tau,s)$, then the proof 
in loc. cit. implies that 
\[
\Psi_{n,r}(v\ot\xi_s)=\Psi_{n,r}(v\ot\xi'_s)
\]
for some $\xi'_s\in I_r(\tau,s)^{R_{r,m}}$. In particular, the integral $\Psi_{n,r}(v\ot\xi_s)=0$ if $\tau$ is not unramified.\\

Next suppose that $r=n$, $\tau$ is unramified and $W_v$ is identically zero on $T_n$. Then it follows from the 
Iwasawa decomposition $H_r=B_{H_n}R_{n,m}$ that
\[
\Psi_{n,n}(v\ot\xi_s)
=
\Psi_{n,n}(v\ot\xi'_s)
=
\int_{T_n} W_v\left(t\right)f_{\xi'_s}(t)\delta_{B_{H_n}}^{-1}(t)dt=0.
\]
Here $B_{H_n}=T_n\ltimes V_n$ stands for the upper triangular Borel subgroup  of $H_n$.
This concludes the proof.
\end{proof}

We also need the following lemma.

\begin{lm}\label{L:converge}
Let $\pi$ (resp. $\tau$) be an irreducible tempered generic representation of $G_n$ (resp. $\GL_r(E)$). Then the 
Rankin-Selberg integrals attached to $\pi$ and $\tau$ converge absolutely for $\Re(s)>0$.
\end{lm}

\begin{proof}
By the Iwasawa decomposition $H_r=B_{H_r}R_{r,0}$, the identification $\hat{A}_r=T_r$ and the right $R_{r,0}$-finiteness, 
we see that each Rankin-Selberg integral can be written as a finite sum of the integrals of the form
\[
\int_{A_r}\int_{{\rm Mat}_{(n-r)\x r}(E)}
W(\t{x}\hat{y})W'(y)\delta^{\frac{1}{2}}_{Q_r}(\hat{y})\delta_{B_{H_r}}^{-1}(\hat{y})|\det(y)|_E^{s-\frac{1}{2}}dxdy
\]
where $W$ (resp. $W'$) is the Whittaker function of an element in $\cV_\pi$ (resp. $\cV_\tau$) with respect to $\psi$
(resp. $\bar{\psi}$), and $B_{H_r}=T_r\ltimes V_r$ is the upper triangular Borel subgroup  of $H_r$. Since 
$\hat{y}^{-1}\t{x}\hat{y}=\widetilde{xy}$, we get from changing the variable that the integral becomes
\[
\int_{A_r}\int_{{\rm Mat}_{(n-r)\x r}(E)}
W(\hat{y}\t{x})W'(a)\delta^{\frac{1}{2}}_{Q_r}(\hat{y})\delta_{B_{H_r}}^{-1}(\hat{y})|\det(y)|_E^{s-\frac{1}{2}+r-n}dxdy.
\]
At this point we invoke \cite[Lemma 4.1]{Ben-ArtziSoudry2009}, which asserts that the function 
\[
x\mapsto W(\hat{y}\t{x})
\]
on ${\rm Mat}_{(n-r)\x r}(E)$ has compact support, which is independent of $y$. It follows that the integral can be 
further written as a finite sum of the integrals of the form (with possibly different $W$)
\[
\int_{A_r}
W(\hat{y})W'(y)\delta^{\frac{1}{2}}_{Q_r}(\hat{y})\delta_{B_{H_r}}^{-1}(\hat{y})|\det(y)|_E^{s-\frac{1}{2}+r-n}dy.
\]
To continue, we parametrize 
\[
y=\diag{y_1\cdots y_r, y_2\cdots y_r,\ldots, y_{r-1}y_r, y_r}
\]
for $y_1,\ldots, y_r$ in $E^\x$ and use asymptotic expansions of Whittaker functions 
(cf. \cite[Section 6]{CasselmanShalika1980}, \cite{LapidMao2009}) to express $W(\hat{y})$ (resp. $W'(\hat{y})$) as a 
finite sum of the functions of the form
\[
\delta_{B_n}^{\frac{1}{2}}(\hat{y})\prod_{j=1}^r\varphi_j(y_j)\chi_j(y_j)
\quad
\text{(resp. $\delta_{B_{\GL_r}}^{\frac{1}{2}}(\hat{y})\prod_{j=1}^r\varphi'_j(y_j)\chi'_j(y_j)$)}
\]
for some locally constant, compact support functions $\varphi_j, \varphi'_j$ on $E$ and finite functions $\chi_j, \chi'_j$ on 
$E^\x$. Here $B_{GL_r}=A_r\ltimes Z_r$ is the upper triangular Borel subgroup of $\GL_r(E)$. Since 
\[
\delta_{B_n}^{\frac{1}{2}}(\hat{y})
\delta_{B_{\GL_r}}^{\frac{1}{2}}(y)
\delta^{\frac{1}{2}}_{Q_r}(\hat{y})
\delta_{B_{H_r}}^{-1}(\hat{y})
|\det(y)|_E^{s-\frac{1}{2}+r-n}
=
|\det(y)|^s
=
\prod_{j=1}^r|y_j|_E^{js}
\]
we are now reducing to estimate the following integral
\[
\prod_{j=1}^r\int_{E^\x}\varphi_j(y_j)\varphi'_j(y_j)\chi_j(y_j)\chi'_j(y_j)|y_j|_E^{js}d^{\x}y_j
\]
which converges absolutely for $\Re(s)\gg 0$ by the temperedness assumption and a result of Waldspurger 
(cf. \cite[Proposition III.2.2]{Waldspurger2003}). This finishes the proof.
\end{proof}

\subsection{Unramified representations}\label{SS:Langlands type}
Because of \lmref{L:vanish}, we shall only consider unramified $\tau$. Then as in \cite{YCheng2022}, we 
assume that $\tau$ is an induced representation of $Langlands'$ $type$. This includes all unramified generic 
representations, but also reducible ones. To describe this type of representations, let 
$\ul{\alpha}=(\alpha_1, \alpha_2\ldots, \alpha_r)$ be an 
$r$-tuple of nonzero complex numbers. Then there is a unique unramified character $\chi_{\ul{\alpha}}$ of $A_r$ given by 
\[
\chi_{\underline{\alpha}}({\rm diag}(y_1, y_2,\ldots, y_r))
=
\alpha_1^{-{\rm log}_{q_E} |y_1|_E}\alpha_2^{-{\rm log}_{q_E}|y_2|_E}\cdots \alpha_r^{-{\rm log}_{q_E}|y_r|_E}.
\]
By extending $\chi_{\ul{\alpha}}$ to a character of the upper trianguler Borel subgroup $B_{\GL_r}=A_r\ltimes Z_r$ of 
$\GL_r(E)$, we can form a normalized induced representation $\tau_{\ul{\alpha}}$ of $\GL_r(E)$. This is an unramified 
representation of $\GL_r(E)$ whose $\GL_r(\frak{o}_E)$-fixed subspace is one-dimensional. Moreover, every unramified 
irreducible representation of $\GL_r(E)$ can be realized as a constituent of $\tau_{\ul{\alpha}}$ for some $\ul{\alpha}$. 
Note that $\tau^*_{\ul{\alpha}}=\tau_{\ul{\alpha}^*}$ with 
$\ul{\alpha}^*:=(\alpha^{-1}_r, \alpha^{-1}_{r-1},\ldots,\alpha_1^{-1})$.\\ 

An unramified representation $\tau$ of $\GL_r(E)$ is called an induced representation of Langlands' type if 
$\tau\cong\tau_{\ul{\alpha}}$ for some $\ul{\alpha}$ with
\begin{equation}\label{E:decreasing}
|\alpha_1|\le |\alpha_2|\le\cdots\le |\alpha_r|.
\end{equation}
These unramified representations may be reducible, but they have the following nice properties that allow us to work 
with: (i) The $\bbC$-linear space ${\rm Hom}_{Z_r}(\tau, \b{\psi}_{Z_r})$ is one-dimensional.
(ii) The intertwining map $u\mapsto W_u$ from $u\in\cV_\tau$ to its associated Whittaker function is an isomorphism
(cf. \cite{JacquetShalika1983},\cite[Lemma 2]{Jacquet2012}). (iii) $\tau^*$ is again induced of Langlands' type. 
We denote by $J(\tau)$ the unique irreducible quotient of $\tau$, which is unramified 
(cf. \cite[Corollary 1.2]{Matringe2013}) with the Satake parameters $\alpha_1,\alpha_2, \ldots,\alpha_r$.\\

Now let $\tau=\tau_{\ul{\alpha}}$ be an unramified representation of $\GL_r(E)$ that is an induced representation 
of Langlands' type. We denote by $\phi_{J(\tau)}$ the $L$-parameter of $J(\tau)$. Then the following lemma compares 
Langlands-Shahidi's local coefficients with arithmetic $\gamma$-factors.

\begin{lm}\label{L:LS=Gal}
We have $c_\delta(s,\tau,{\rm As},\psi_F)=\gamma(s,\phi_{J(\tau)},{\rm As},\psi_F)$.
\end{lm}

\begin{proof}
This follows from the multiplicativity of $c_\delta(s,\tau,{\rm As}, \psi_F)$ and $\gamma(s,\phi_{J(\tau)},{\rm As},\psi_F)$, as 
well as the choice of $\delta$ and $\psi_F$ (cf. \cite[Proposition 2.3.1]{Shankman}).
\end{proof}

To state the next lemma, we fix a non-zero element $v_\tau\in\cV_\tau^{\GL_r(\frak{o}_E)}$. Then the 
decomposition $H_r=Q_rR_{r,m}$ implies that the space $I_r(\tau,s)^{R_{r,m}}$ is one-dimensional, so we can 
fix a basis $\xi_{\tau,s}^m$ with $\xi_{\tau,s}^m(I_{2r})=v_\tau$.

\begin{lm}\label{L:GK method for m}
We have 
\[
M(\tau, s)\xi^m_{\tau,s}
=
\omega_{\tau_s}(\varpi)^m\cdot \frac{L(2s-1,\phi_{J(\tau)},{\rm As})}{L(2s, \phi_{J(\tau)}, {\rm As})}\cdot
\xi^m_{\tau^*,1-s}
\]
where $\omega_{\tau_s}$ is the central character of $\tau_s$.
\end{lm}

\begin{proof}
The proof of this lemma is similar to that of \cite[Lemma 6.4]{YCheng2022}. Indeed, it is slightly simpler 
then the one in loc. cit. as we do have additional Weyl elements here.
\end{proof}

\section{Rankin-Selberg integrals attached newforms and oldforms}\label{S:proof of main2}
We shall prove \thmref{T:main2} in this section. Let us begin with the setups.

\subsection{Setups}\label{SS:setup}
As before, let $\pi$ be an irreducible generic representation of $G_n$ with the associated $L$-parameter 
$\phi_\pi:WD_E\to\GL_{2n+1}(\mathbb{C})$.  Fix a non-zero $\Lambda_{\pi,\psi_E}\in{\rm Hom}_{N_n}(\pi,\psi_{N_n})$
and a newform $v_\pi\in\cV_\pi^{K_{n,a_\pi}}$. Let $\ul{\alpha}=(\alpha_1,\alpha_2,\ldots,\alpha_r)$ be a $r$-tuple of 
non-zero complex numbers and $\ul{\dot{\alpha}}$ be its rearrangement so that \eqref{E:decreasing} holds. 
Let $\tau=\tau_{\ul{\dot{\alpha}}}$, which is an unramified representation of $\GL_r(F)$ that is an induced representation 
of Langlands' type (cf. \S\ref{SS:Langlands type}). Let $J(\tau)$ be the unique unramified quotient of $\tau$ and 
$\phi_{J(\tau)}$ be the $L$-parameter of $J(\tau)$. Fix non-zero elements $v_\tau\in\cV_\tau^{\GL_r(\frak{o})}$ and 
$\Lambda_{\tau,\b{\psi}_E}\in{\rm Hom}_{Z_r}(\tau,\b{\psi}_{Z_r})$. Since $\psi_E$ is unramified, we may assume 
\[
\Lambda_{\tau,\b{\psi}_E}(v_\tau)=1
\]
(cf. \cite{Shintani1976}, \cite{Jacquet2012}). Let $\xi^m_{\tau,s}$ be the basis of 
the one-dimensional space $I_r(\tau,s)^{R_{r,m}}$ with $\xi_{\tau,s}^m(I_{2r})=v_\tau$.\\ 

The Haar measures appeared in the Rankin-Selberg integrals are chosen as follows.
First, we take $dx$ to be the Haar measure on ${\rm Mat}_{(n-r)\x r}(E)$ with 
${\rm vol}({\rm Mat}_{(n-r)\x r}(\frak{o}_E),dx)=1$.
On the other hand, the Haar measures $dt$ on $T_r$ and $dk$ on $R_{r,m}$ are chosen so that 
${\rm vol}(T_r\cap R_{r,m},dt)={\rm vol}(R_{r,m},dk)=1$.
Note that $T_r\cap R_{r,m}$ is independent of $m$. Then the quotient measure $dh$ on $V_r\backslash H_r$ is given by 
\[
\int_{V_r\backslash H_r}f(h)dh
=
\int_{T_r}\int_{R_{r,m}}f(tk)\delta_r^{-1}(t)dtdk
\]
where $\delta_r$ denotes the modulus function of the upper triangular Borel subgroup $T_r\ltimes V_r$ of $H_r$.
We indicate that \lmref{L:same vol} implies that $dh$ does not depend on $m$.

\subsection{Preliminaries}\label{SS:pre2}
Let $v\in\cV_\pi^{K_{n,m}}$ and $W_v$ be the associated Whittaker function with respect to 
$\psi_E$. Then since $W_v$ is right $K_{n,m}$-invariant, its not hard to see that the support of $W_v$ when restricting to 
$T_n$ is contained in 
\[
\stt{{\rm diag}(y_1,\ldots, y_n, 1, \bar{y}_n^{-1},\ldots,\bar{y}_1^{-1})\mid 0<|y_1|_E\le |y_2|_E\le\cdots\le |y_n|_E\le 1}.
\]
In the following two lemmas, we will reveal more properties of $W_v$.

\begin{lm}\label{L:supp of W_v}
Let $a\in\GL_r(F)$ and $x\in{\rm Mat}_{(n-r)\x r}(E)\setminus{\rm Mat}_{(n-r)\x r}(\frak{o}_E)$ with $r<n$.
Then we have $W_v(\hat{a}\tilde{x})=0$.
\end{lm}

\begin{proof}
We refer to \S\ref{SSS:embed} and \eqref{E:X_n,r} for the notation of $\hat{a}$ and $\tilde{x}$ respectively.
The proof of this lemma the same as \cite[Lemma 6.5]{YCheng2022}. Observe that the matrices used in loc. cit. 
correspond to the root matrices $\chi_{\e_k-\e_{\ell+r+1}}(y)$ and $\chi_{\e_k}(y)$ here.
\end{proof}

The proof of the next lemma is inspired by the proof of \cite[Theorem 4.4]{AtobeOiYasuda} which is one of the keys 
for showing that $\Lambda_{\pi,\psi_E}(v_\pi)\ne 0$ when $\pi$ is tempered.

\begin{lm}\label{L:key lemma}
Suppose that $\pi$ is tempered and $v$ is non-zero. Then $W_v$ can not be identically zero on $T_n$.
\end{lm}

\begin{proof}
The proof is by contradiction. So let us suppose that $W_v$ is identically zero on $T_n$. Let 
$\langle\cdot,\cdot\rangle_\pi$ be an $G_n$-equivalent Hermitian pairing on $\cV_\pi$ and 
$f_v(g)=\langle \pi(g)v,v\rangle_\pi$ be a matrix coefficient of $\pi$. Since $v\ne 0$, we have $f_v(I_{2n+1})\ne 0$.
Then from the proof of \cite[Lemma 12.5]{GanSavin2012}), there exists an 
irreducible tempered representation $\pi'$ of $H_n$ and a matrix coefficient $f'$ of $\pi'$ such that 
\begin{equation}\label{E:nonzero}
\int_{H_n} f_v(h)f'(h)dh\ne 0.
\end{equation}
Moreover, as argued in the proof of \cite[Theorem 4.4]{AtobeOiYasuda}, this $\pi'$ must be unramified and $f'$ can be 
chosen to be bi-$R_{n,m}$-invariant. It follows that $\pi'=\rho_{\tau,1/2}$ (cf. \S\ref{SS:ind rep}) for some irreducible
unramified tempered representation $\tau$ of $\GL_n(E)$. Note that $\tau$ is necessarily generic and we may 
assume $f'(h)=\langle\pi'(h)\xi_{\tau,1/2}^m, \xi^m_{\tau,1/2}\rangle_{\pi'}$, where $\langle\cdot,\cdot\rangle_{\pi'}$ stands 
for an $H_n$-equivalent Hermitian pairing on $I_n(\tau,1/2)$.\\

Now the proof follows immediately from \lmref{L:vanish} and \lmref{L:converge}. Indeed, \eqref{E:nonzero} 
implies 
\[
v'\ot\xi_{1/2}\mapsto \int_{H_n}\langle \pi(h)v',v\rangle_\pi\langle\pi'(h)\xi_{1/2}, \xi_{\tau,1/2}^m\rangle_{\pi'}dh
\]
defines a non-zero element in ${\rm Hom}_{H_n}(\pi\ot\pi',\mathbb{C})$. On the other hand, by 
\lmref{L:converge} and \cite[Proposition 4.9]{Ben-ArtziSoudry2009}, 
\[
v'\ot\xi_{1/2}\mapsto \Psi_{n,n}(v'\ot\xi_{1/2})
\]
also defines a non-zero element in the same Hom space. It then follows from the multiplicity one result 
(cf. \cite{AGRS2010}) and \lmref{L:vanish} that
\[
\int_{H_n}f_v(h)f'(h)dh=c\Psi_{n,n}(v\ot\xi_{\tau,1/2})=0
\]
where $c$ is a non-zero constant. This contradicts to \eqref{E:nonzero}; hence the result.
\end{proof}

The following proposition, whose proof is similar to that of \cite[Proposition 6.7]{YCheng2022}, is the key for computing 
the Rankin-Selberg integrals and is where we need the Assumption \ref{H}, namely, we assume that 
\[
\gamma_\delta(s,\pi\x\tau,\psi_F)=\gamma(s,\phi_\pi\ot\phi_{J(\tau)},\psi_E).
\]
To describe the proposition, let $\cS_r$ be the $\bbC$-algebra of symmetric polynomials in 
\[
(X_1, X_1^{-1}, X_2, X_2^{-1},\ldots, X_r, X_r^{-1})
\]
and $\cS_r^0$ be its subalgebra consisting of elements $f$ satisfying
\[
f(X_1,\ldots, X^{-1}_j,\ldots, X_r)
=
f(X_1,\ldots, X_j,\ldots, X_r)
\]
for $1\le j\le r$. For each $m$, we have the algebra isomorphism (Satake isomorphism)
\[
\sS^0_{r,m}: \sH(H_r//R_{r,m})\overset{\varsigma_{r,m}^0}{\longto}\sH(T_r//T_r\cap R_{r,m})^{\sW_{H_r}}
\cong\bbC[A_{r,\bbC}]^{\frak{S}_r\ltimes(\bbZ/2\bbZ)^{r}}=\cS_r^0
\]
where $A_{r,\bbC}$ is the diagonal torus of $\GL_r(\bbC)$ and
\[
\varsigma^0_{r,m}(\varphi)(t)
=
\delta^{\frac{1}{2}}_{B_{H_r}}(t)\int_{V_r}\varphi(tv)dv
\quad
({\rm vol}(V_r(\frak{o}),dv)=1).
\]
Finally, let $\sS_{r,m}:\sH(H_r//R_{r,m})\to\cS_r$ be the (injective) algebra homomorphism obtained by composing 
$\sS^0_{r,m}$ with the natural embedding $\cS_r^0\hookto\cS_r$. Now the proposition can be stated as follows.

\begin{prop}\label{P:main prop}
There exist linear maps
\[
\Xi^m_{n,r}: \cV_{\pi}^{K}\longto\cS_r;
\quad
v\mapsto\Xi^m_{n,r}(v;X_1,\ldots,X_r)
\]
where $K=K_{n,m}$ if $r<n$ and $K=R_{n,m}$ if $r=n$, such that the followings are satisfied:\\

\begin{itemize}
\item[(1)]
We have
\[
\Xi^m_{n,r}(v;q_E^{-s+1/2}\alpha_1, \ldots,q_E^{-s+1/2}\alpha_r)
=
\frac{L(2s,\phi_{J(\tau)},{\rm As})\Psi_{n,r}(v\ot \xi^m_{\tau,s})}{L(s,\phi_\pi\ot\phi_{J(\tau)})}
\]
for every $s\in\bbC$.
\item[(2)]
The functional equation
\[
\Xi^m_{n,r}(v; X_1^{-1}, \ldots, X_r^{-1})=(X_1\cdots X_r)^{(a_\pi-m)}
\Xi^m_{n,r}(v;X_1,\ldots,X_r)
\]
holds.
\item[(3)]
The kernel of $\Xi^m_{n,r}$ is given by 
\[
ker(\Xi^m_{n,r})=\stt{v\in \cV_\pi^K\mid \text{$W_v(t)=0$ for every $t\in T_r$}}.
\]
\item[(4)]
The relation 
\[
\Xi^m_{n,r}(v;X_1, \ldots,X_{r-1},0)=\Xi^m_{n,r-1}(v;X_1,\ldots,X_{r-1})
\]
holds for $v\in\cV_\pi^{K_{n,m}}$ and $2\leq r\leq n$.
\item[(5)]
We have
\[
\Xi^m_{n,n}(\varphi\star v;X_1,\ldots, X_n)
=\sS_{n,m}(\varphi)\cdot \Xi^m_{n,n}(v; X_1,\ldots, X_n)
\]
for $\varphi\in\sH(H_n//R_{n,m})$. 
\end{itemize}
\end{prop}

\begin{proof}
As mentioned, the proof is similar to that of \cite[Proposition 6.7]{YCheng2022}. So here we merely give a sketch of 
the proof. Let $v\in\cV_\pi^K$ and $W(a;\alpha_1,\ldots,\alpha_r;\b{\psi}_E):=\Lambda_{\tau,\b{\psi}_E}(\tau(a)v_\tau)$ 
for $a\in\GL_r(E)$. Note that the scalar $W(a;\alpha_1,\ldots,\alpha_r;\b{\psi}_E)$ is the value of the polynomial 
$W(a;X_1,\ldots, X_r; \b{\psi}_E)\in\cS_r$ at $(\alpha_1,\ldots, \alpha_r)$ (cf. \cite[Section 2]{Jacquet2012}, see also 
\cite[Section 5.3]{YCheng2022}). Then by \lmref{L:supp of W_v} and the choice of the measures, we have
\[
\Psi_{n,r}(v\ot\xi^m_{\tau,s})
=
\sum_{\ell\in\bbZ}\int_{a\in Z_r\backslash\GL_r(E),\,|\det(a)|_E=q_E^{-\ell}}
W_v(\hat{a})W(a;\alpha_1,\ldots,\alpha_r;\b{\psi}_E)|\det(a)|_E^{s-\frac{1}{2}+n-\frac{r}{2}}da
\]
for $\Re(s)\gg 0$. This is derived in a similar way as in the proof of \cite[Lemma 6.6]{YCheng2022}. 
Observe that the power of $|\det(a)|$ is different from that of loc. cit. due the the difference between the modulus functions 
of $Q_r$. Now consider the following integral 
\[
\Psi^m_{n,r,\ell}(v; X_1,\ldots, X_r)
=
\int_{a\in Z_r\backslash\GL_r(E),\,|\det(a)|_E=q_E^{-\ell}}
W_v(\hat{a})W(a;X_1, \ldots,X_r;\b{\psi}_E)|\det(a)|_E^{\frac{r}{2}-n}da
\]
which is indeed a finite sum; hence gives rise to a homogeneous polynomial of degree $\ell$ in $\cS_r$. 
Furthermore, there exists an integer $N_{v}$ depending only on 
$v$ such that $\Psi^m_{n,r,\ell}(v; X_1,\ldots, X_r)=0$ for all $\ell<N_v$. We thus define the following formal 
Laurent series 
\[
\Psi^m_{n,r}(v;X_1,\ldots, X_r;Y)
=
\sum_{\ell\in\bbZ}\Psi^m_{n,r,\ell}(v;X_1,\ldots, X_r)Y^\ell
=
\sum_{\ell\ge N_{v}}\Psi^m_{n,r,\ell}(v;X_1,\ldots, X_r)Y^\ell
\]
with coefficient in $\cS_r$. It then follows from the construction that 
\[
\Psi^m_{n,r}(v;\alpha_1,\ldots,\alpha_r; q_E^{-s+\frac{1}{2}})
=
\Psi_{n,r}(v\ot\xi^m_{\tau,s})
\]
for $\Re(s)\gg 0$.\\

Next, let $P_{\phi_\pi}(Y)\in\bbC[Y]$ be such that $L(s,\phi_\pi)=P_{\phi_\pi}(q_E^{-s})^{-1}$ and put
\[
P_{\phi_\pi}(X_1,\ldots, X_r; Y)
=
\prod_{j=1}^r P_{\phi_\pi}(q_E^{-\frac{1}{2}}X_jY).
\]
Since $\phi_\pi$ is conjugate self-dual, \corref{C:L-factor} gives
\[
L(s,\phi_\pi\ot\phi_{J(\tau)})
=
P_{\phi_\pi}(\a_1,\ldots,\a_r;q_E^{-s+\frac{1}{2}})^{-1}
\quad\text{and}\quad
L(1-s,\t{\phi}_\pi\ot\phi_{J(\tau^*)})
=
P_{\phi_\pi}(\a^{-1}_1,\ldots, \a^{-1}_r;q_E^{s-\frac{1}{2}})^{-1}.
\]
On the other hand, let
\[
P_{{\rm As}}(X_1,\ldots,X_r; Y)
=
\prod_{1\le i<j\le r}(1-q_E^{-1}X_iX_jY^2)\prod_{1\le k\le r}(1-q_E^{-\frac{1}{2}}X_kY).
\]
Then by definition
\[
L(2s,\phi_{J(\tau)},{\rm As})
=
P_{{\rm As}}(\a_1,\ldots,\a_r; q_E^{-s+\frac{1}{2}})^{-1}
\quad\text{and}\quad
L(2-2s,\phi_{J(\tau^*)},{\rm As})
=
P_{{\rm As}}(\a^{-1}_1,\ldots,\a^{-1}_r; q_E^{s-\frac{1}{2}})^{-1}.
\]
Now we define
\[
\Xi^m_{n,r}(v;X_1,\ldots,X_r;Y)
=
\frac{P_{\phi_\pi}(X_1,\ldots, X_r; Y)\Psi^m_{n,r}(v;X_1,\ldots,X_r;Y)}
{P_{{\rm As}}(X_1,\ldots,X_r; Y)}
\]
which is again a formal Laurent series with coefficients in $\cS_r$, and we have
\[
\Xi^m_{n,r}(v;\a_1,\ldots,\a_r; q_E^{-s+\frac{1}{2}})
=
\frac{L(2s,\phi_{J(\tau)},{\rm As})\Psi_{n,r}(v\ot\xi^m_{\tau,s})}{L(s,\phi_\pi\ot\phi_{J(\tau)})}
\]
for $\Re(s)\gg 0$ from the constructions. Moreover, by replacing $\alpha_1,\ldots,\alpha_r$ and $s$ with 
$\alpha_1^{-1},\ldots,\alpha_r^{-1}$ and $1-s$ respectively, we see that
\[
\Xi^m_{n,r}(v;\a^{-1}_1,\ldots,\a^{-1}_r; q_E^{s-\frac{1}{2}})
=
\frac{L(2-2s,\phi_{J(\tau^*)},{\rm As})\Psi_{n,r}(v\ot\xi^m_{\tau^*,1-s})}{L(1-s,\phi_\pi\ot\phi_{J(\tau^*)})}
\]
for $\Re(s)\ll 0$.\\

To connect two formal Laurent series $\Xi^m_{n,r}(v;X_1,\ldots, X_r; Y)$ and 
$\Xi^m_{n,r}(v;X^{-1}_1,\ldots, X^{-1}_r; Y^{-1})$, we have to apply the functional equation \eqref{E:FE} and this is 
the place that we need the Assumption \ref{H}. More precisely, \lmref{L:LS=Gal} and \lmref{L:GK method for m} imply
\[
M^\dag_{\psi_F,\delta}(\tau, s)\xi^m_{\tau,s}
=
\omega_{\tau_s}(\varpi)^m\cdot \frac{L(2-2s,\phi_{J(\tau^*)},{\rm As})}{L(2s, \phi_{J(\tau)}, {\rm As})}\cdot
\xi^m_{\tau^*,1-s}.
\]
Then the functional equation \eqref{E:FE} and the Assumption \ref{H} give
\[
\frac{L(2-2s,\phi_{J(\tau^*)},{\rm As})\Psi_{n,r}(v\ot\xi^m_{\tau^*,1-s})}{L(1-s,\t{\phi}_\pi\ot\phi_{J(\tau^*)})}
=
\omega_{\tau_s}(\varpi)^{-m}\epsilon(s,\phi_\pi\ot\phi_{J(\tau)},\psi)
\frac{L(2s,\phi_{J(\tau)},{\rm As})\Psi_{n,r}(v\ot\xi^m_{\tau,s})}{L(s,\phi_\pi\ot\phi_{J(\tau)})}.
\]
Now if we put
\[
\epsilon_{\phi_\pi,m}(X_1,\ldots,X_r;Y)
=
(X_1\cdots X_r)^{a_\pi-m}Y^{(a_\pi-m)r}
\]
then 
\[
\epsilon_{\phi_\pi,m}(\a_1,\ldots,\a_r;q^{-s+\frac{1}{2}})
=
\omega_{\tau_s}(\varpi)^{-m}\epsilon(s,\phi_\pi\ot\phi_{J(\tau)},\psi)
\]
for $s\in\bbC$ (cf. \eqref{E:epsilon}) and the relations above imply
\[
\Xi^m_{n,r}(v;X_1^{-1},\ldots, X_r^{-1}; Y^{-1})
=
\epsilon_{\phi_\pi,m}(X_1,\ldots, X_r;Y)
\Xi^m_{n,r}(v;X_1,\ldots,X_r;Y).
\]
In particular,  $\Xi^m_{n,r}(v;X_1,\ldots,X_r;Y)$ is a finite sum; hence it makes sense to define
\[
\Xi^m_{n,r}(v;X_1,\ldots,X_r)
=
\Xi^m_{n,r}(v;X_1,\ldots,X_r;1)\in\cS_r.
\]
Furthermore, the same reasoning in the proof of \cite[Proposition 6.7]{YCheng2022} gives
\[
\Xi^m_{n,r}(v;YX_1,\ldots, YX_r)=\Xi^m_{n,r}(v;X_1,\ldots,X_r;Y).
\]
Together, the existence of the linear maps $\Xi_{n,r}^m$ and the first assertion are proved. The proofs of the rest of 
the assertions are similar to that of loc. cit. We point out that when $\pi$ is unramified and $v\in\cV_\pi^{K_{n,0}}$, 
the assertion $(4)$ is essentially obtained in the proof of \cite[Theorem 8.1]{Ben-ArtziSoudry2009}. However, same 
argument actually applies even if $\pi$ is not unramified and $v\in\cV_\pi^{K_{n,m}}$ for $m>a_\pi$.
\end{proof}

\subsection{Proof of \thmref{T:main2}}
Now we are ready to verify \thmref{T:main2}. We claim that 
\begin{equation}\label{E:key id}
\Xi^{a_\pi}_{n,n}(v_\pi; X_1, \ldots, X_n)
=
\Lambda_{\pi,\psi_E}(v_\pi).
\end{equation}
Then the proof follows. In fact, \eqref{E:main eqn} will be implied by \propref{P:main prop} $(1)$ and $(4)$. On the other
hand, if $\pi$ is tempered, then \lmref{L:key lemma} and \propref{P:main prop} $(3)$ will tell us that 
$\Lambda_{\pi,\psi_E}(v_\pi)\ne 0$.\\

To prove \eqref{E:key id}, let $Y_j$ be the $j$-th elementary symmetric 
polynomial in $X_1, \ldots, X_n$ for $1\le j\le n$. Then we have 
\[
\cS_n=\bbC[Y_1, \ldots, Y_{n-1}, Y^{\pm}_n].
\] 
Note that $Y_n=X_1X_2\cdots X_n$ gives a $\bbZ$-grading $\cS_n$ by the degree of $Y_n$:
\[
\cS_n=\bigoplus_{\ell\in\bbZ}\bbC[Y_1,\ldots, Y_{n-1}]\, Y_n^\ell.
\] 
Now the key observation is that
\begin{equation}\label{E:degree in Y_n}
\Xi^m_{n,n}(v;X_1,\ldots, X_n)\in\bigoplus_{\ell\ge 0}\cS_{n,\ell}
\end{equation}
for every $v\in\cV_\pi^{K_{n,m}}$ and $m\ge 0$. This follows from the property of $W_v$ remarked in the beginning of 
\S\ref{SS:pre2} and the fact that
\[
{\rm deg}_{Y_n}W(\diag{\varpi^{\lambda_1},\ldots,\varpi^{\la_n}};X_1,\ldots, X_n;\b{\psi}_E)=\la_n
\]
where $\la_1\ge\cdots\ge\la_n$ are integers. Then \propref{P:main prop} $(2)$ gives
\[
\Xi^{a_\pi}_{n,n}(v_\pi;X^{-1},\ldots, X_n^{-1})
=
\Xi^{a_\pi}_{n,n}(v_\pi;X^{-1},\ldots, X_n^{-1})
=
\Xi^{a_\pi}_{n,n}(v_\pi; X_1,\ldots, X_n);
\]
hence by \eqref{E:degree in Y_n}, we must have $\Xi^{a_\pi}_{n,n}(v_\pi; X_1, \ldots, X_n)=c$ for some constant $c$.\\

To complete the proof, it remains to show that $c=\Lambda_{\pi,\psi_E}(v_\pi)$. By \propref{P:main prop} $(1)$, we find 
that 
\begin{equation}\label{E:c}
cL(s,\phi_\pi\ot\phi_{J(\tau)})=L(2s,\phi_{J(\tau)},{\rm As})\Psi_{n,n}(v_\pi\ot \xi^{a_\pi}_{\tau,s}).
\end{equation}
Then since 
\[
\Psi_{n,n}(v\ot\xi^{a_\pi}_{\tau,s})
=
\sum_{\ell\ge 0}\int_{a\in A_n,\,|\det(a)|_E=q_E^{-\ell}}
W_v(\hat{a})W(a;\alpha_1,\ldots,\alpha_r;\b{\psi}_E)\delta_{B_{H_n}}^{-1}(\hat{a})\delta_{Q_n}^{\frac{1}{2}}(\hat{a})
|\det(a)|_E^{s-\frac{1}{2}}da
\]
by the Iwasawa decomposition $H_n=B_{H_n}R_{n,a_\pi}$, the identification $T_n=\hat{A}_n$ and again the property of 
$W_{v_\pi}$ mentioned in the beginning of \S\ref{SS:pre2}, we see that the constant term in the LHS of \eqref{E:c}, 
when writing both sides as power series in $q_E^{-s}$, is $c$; while that in the RHS of \eqref{E:c} is 
$\Lambda_{\pi,\psi_E}(v_\pi)$. Note that here we also use the fact that $W(I_n;\a_1,\ldots,\a_n;\b{\psi}_E)=1$.
This finishes the proof.\qed

\subsection{Rankin-Selberg integrals attached to oldforms}
In this paper, we also compute the Rankin-Selberg integrals attached to oldforms. We in fact compute the integrals 
attached to elements in the conjectural basis $\sB_{\pi,m}$ of $\cV_\pi^{K_{n,m}}$ proposed in \S\ref{SS:basis}. 
The key for the computations is again \propref{P:main prop}, but additional works are required. Recall the notation
\[
\mu_\ell=\underbrace{(\ell,\ldots,\ell)}_{n}\in\sP
\quad\text{and}\quad
\varpi^{\mu_\ell}=\diag{\varpi^\ell I_n,\varpi^{-\ell} I_n}\in T_n
\]
and the convention $\varpi^{-\mu_\ell}=\left(\varpi^{\mu_\ell}\right)^{-1}$. Now let $v\in\cV_\pi^{R_{n,m}}$. Then 
$\pi(\varpi^{-\mu_1})v\in\cV_\pi^{R_{n,m+2}}$ and we have the following lemma.

\begin{lm}\label{L:eta action}
Let $v\in\cV_\pi^{R_{n,m}}$. We have 
\[
\Xi^{m+2}_{n,n}(\pi(\varpi^{-\mu_1})(v);X_1,\ldots, X_n)
=
q_E^{\frac{n(n-1)}{2}}(X_1\cdots X_n)\Xi^m_{n,n}(v;X_1,\ldots, X_n).
\]
\end{lm}

\begin{proof}
This follows from the identical arguments in the proof of \cite[Lemma 7.1]{YCheng2022}.
\end{proof}

Recall that $v_\pi'\in\cV_\pi^{K_{n,a_\pi+1}}$ stands for a basis. The next lemma compute the integrals attached to 
$v_\pi'$. To state and to prove it, we denote by $\cY_j(X_1, X_2,\ldots, X_r)$ the $j$-th elementary 
symmetric polynomial in $X_1, X_2,\ldots, X_r$ for $1\le j\le r$ and put $\cY_0(X_1, X_2,\ldots, X_r)=1$. 
Then the relations 
\[
\cY_j(X_1,X_2,\ldots, X_{r-1},0)=\cY_j(X_1,X_2,\ldots, X_{r-1})
\quad\text{and}\quad
\cY_r(X_1,X_2,\ldots, X_{r-1},0)=0
\]
hold for $0\le j\le r-1$.

\begin{lm}\label{L:level a+1}
We have
\[
\Xi_{n,n}^{a_\pi+1}(v'_\pi;X_1,\ldots, X_n)=\Lambda_{\pi,\psi_E}(v'_\pi)\sum_{j=0}^n \cY_j(X_1,\ldots,X_n).
\]
Moreover, if $\pi$ is tempered, then $\Lambda_{\pi,\psi_E}(v'_\pi)\ne 0$.
\end{lm}

\begin{proof}
By \eqref{E:degree in Y_n}, we can write
\begin{equation}\label{E:rought Xi for level a+1}
\Xi^{a_\pi+1}_{n,n}(v_\pi';X_1,\ldots, X_n)
=
\sum_{\ul{\ell}}
b_{\ul{\ell}}
\cY_1(X_1,\ldots, X_n)^{\ell_1}\cY_2(X_1,\ldots, X_n)^{\ell_2}\cdots\cY_n(X_1,\ldots, X_n)^{\ell_n}
\end{equation}
for some $\ul{\ell}=(\ell_1,\ldots, \ell_n)\in\bbZ^n_{\ge 0}$ and $b_{\ul{\ell}}\in\bbC$. Then since 
\[
\cY_j(X^{-1}_1,\ldots, X_r^{-1})=\cY_{r-j}(X_1,\ldots, X_r)\cY_r(X_1,\ldots, X_r)^{-1}
\]
for $0\le j\le r$, the functional equation in \propref{P:main prop} $(2)$ gives
\[
\Xi^{a_\pi+1}_{n,n}(v'_\pi;X_1,\ldots, X_n)
=
\sum_{\ul{\ell}}
b_{\ul{\ell}}\cY_1(X_1,\ldots, X_n)^{\ell_{n-1}}\cY_2(X_1,\ldots, X_n)^{\ell_{n-2}}\cdots 
\cY_n(X_1,\ldots, X_n)^{1-\ell_1-\ell_2-\cdots-\ell_n}.
\]
As $1-\ell_1-\ell_2-\cdots-\ell_n\ge 0$, we find that $\ell_j\le 1$ for $1\le j\le n$; hence 
\eqref{E:rought Xi for level a+1} becomes
\[
\Xi^{a_\pi+1}_{n,n}(v'_\pi;X_1,\ldots, X_n)
=
\sum_{j=0}^n b_j\cY_j(X_1,\ldots, X_n).
\]
To preceed, we apply \propref{P:main prop} (4) to get
\[
\Xi^{a_\pi+1}_{n,r}(v;X_1,\ldots, X_r)
=
\sum_{j=0}^r b_j\cY_j(X_1,\ldots, X_r)
\]
for $1\le r\le n$. Then the functional equation for $\Xi^{a_\pi+1}_{n,r}(v'_\pi;X_1,\ldots, X_r)$ implies
\[
\Xi^{a_\pi+1}_{n,r}(v'_\pi;X_1,\ldots, X_r)
=\cY_r(X_1,\ldots,X_r)\Xi_{n,r}^{a_\pi+1}(v'_\pi;X^{-1}_1,\ldots,X^{-1}_r)
=
\sum_{j=0}^rb_j\cY_{r-j}(X_1,\ldots,X_r).
\]
By comparing the constant terms, we find that $b_j=b_0$ for $1\le j\le n$; thus 
\[
\Xi^{a_\pi+1}_{n,n}(v'_\pi;X_1,\ldots, X_n)
=
b_0\sum_{j=0}^n\cY_j(X_1,\ldots, X_n).
\] 
Finally, exactly the same arguments as in the proof of \thmref{T:main2} above show that 
$b_0=\Lambda_{\pi,\psi_E}(v'_\pi)$ and $\Lambda_{\pi,\psi_E}(v'_\pi)\ne 0$ if $\pi$ is tempered. This completes the proof.
\end{proof}

Now we can describe how to compute Rankin-Selberg integrals attached to elements in $\sB_{\pi,m}$.

\begin{thm}\label{T:main for oldform}
Let  $\sB_{\pi,m}$ be the set defined in \S\ref{SS:basis} for each $m\ge a_\pi$. Then the Rankin-Selberg integrals 
$\Psi_{n,r}(v\ot\xi^m_{\tau,s})$ attached to $v\in\sB_{\pi,m}$ and $\xi^m_{\tau,s}$ can be computed by using 
\propref{P:main prop} $(1)$, $(4)$ and $(5)$ together with \lmref{L:eta action}, \lmref{L:level a+1} and 
\eqref{E:main eqn}.
\end{thm}

\begin{proof}
The proof is almost clear from the statement. However, there is one point that we need to clarify, namely, in the definition of 
$\sB_{\pi,m}$, the Hecke algebra elements involved are in $\sH(H_n//R_{n,e})$ with $e=0,1$;  while in 
\propref{P:main prop} $(5)$, the Satake isomorphisms $\sS_{n,m}$ are for any $m\ge 0$. To explain this, let $a=a_\pi$ or 
$a_\pi+1$ and $v_0\in\cV_\pi^{K_{n,a}}$ be a basis. Write $a=e+2\ell$ for some $e=0,1$ and $\ell\ge 0$. Let $m\ge a$ 
be an integer whose parity is the same as $a$, so that $m=e+2\ell'$ for some $\ell'\ge\ell$. Then $\sB_{\pi,m}$ consists 
of the elements of the form
\[
v=\pi(\varpi^{-\mu_{\ell'}})\varphi_{\la,e}\star \pi(\varpi^{\mu_\ell})(v_0)
\]
for some $\la\in\sP$. Now a straightforward computation shows that 
\[
v=\pi(\varpi^{-\mu_{\ell'-\ell}})\varphi_{\la,a}\star v_0;
\]
hence \propref{P:main prop} $(5)$ and \lmref{L:eta action} give
\begin{equation}\label{E:right eqn}
\Xi^m_{n,n}(v;X_1,\ldots, X_n)
=
q_E^{\frac{n(n-1)+(m-a)}{2}}(X_1\cdots X_n)^{\frac{m-a}{2}}\sS_{n,a}(\varphi_{\la,a})
\cdot
\Xi^a_{n,n}(v_0;X_1,\ldots, X_n).
\end{equation}
This concludes the proof.
\end{proof}

As a corollary, we see that $\sB_{\pi,m}$ defines a basis of $\cV_\pi^{K_{n,m}}$ for each $m> a_\pi+1$ when $\pi$
is tempered.

\begin{cor}
Suppose that $\pi$ is tempered. Then under the Assumption \ref{H}, the set $\sB_{\pi,m}$ defines a basis of 
$\cV_\pi^{K_{n,m}}$ for each $m> a_\pi+1$.
\end{cor}

\begin{proof}
This follows immediately from \eqref{E:main eqn}, \lmref{L:level a+1} and \eqref{E:right eqn}.
\end{proof}

\section*{Appendix}\label{A}
The aim of this appendix is to make sense the functional equation \eqref{E:FE} even when the representations involved 
are reducible. So let $\psi$ be a non-trivial additive character of $E$ and $\pi$ (resp. $\tau$) be a representation of $G_n$ 
(resp. $\GL_r(E)$) that is of Whittaker type. Let $Z'$ and $Y'$ be subgroups of $G_n$ given by 
\[
Z'
=
\stt{
z'
=
\begin{pmatrix}
I_r&&&&\\
&z&&&\\
&&1&&\\
&&&z^*&\\
&&&&I_r
\end{pmatrix}
\in G_n
\mid
z\in Z_{n-r}
}
\quad\text{and}\quad
Y'
=
\stt{
y'
=
\begin{pmatrix}
I_r&0&0&b&0\\
a&I_{n-r}&x&c&b'\\
&&1&x'&0\\
&&&I_{n-r}&0\\
&&&a'&I_r
\end{pmatrix}
\in G_n
}.
\]
Put $Y=Z'Y'$ which is a subgroup of $G_n$; we define a character $\psi_Y$ of $Y$ by 
$\psi_Y(z'y')=\psi_{Z_{n-r}}(z)\psi(x_{n-r})$, where ${}^tx=(x_1,\ldots, x_{n-r})\in E^{n-r}$. Note that 
$\psi_Y(hyh^{-1})=\psi_Y(y)$ for $h\in H_r$ and $y\in Y$ (recall the embedding $H_r\hookto G_n$ in 
\S\ref{SSS:embed}). Thus $\psi_Y$ extends to a character of $H_r\ltimes Y\subset G_n$ which we again denoted by 
$\psi_Y$. Now one checks that the Rankin-Selberg integral \eqref{E:RS int} gives rise to an element in
\begin{equation}\label{E:hom}
{\rm Hom}_{H_r\ltimes Y}(\pi|_{H_r\ltimes Y}\ot\rho_{\tau,s},\psi_Y)
\end{equation}
for $\Re(s)\gg 0$ by the integral and by meromorphic continuation in general.
Here we let $Y$ acts trivially on $\rho_{\tau,s}$. Then we prove the following:

\begin{thmA}
The dimension of the Hom space \eqref{E:hom} is equal to one, except for finitely many values of $q_E^{-s}$. 
\end{thmA}

\begin{proof}
The proof is obtained by adopting the method of Soudry, which porves a similar result for the Rankin-Selberg integrals 
attached to ${\rm SO}_{2n+1}\x\GL_r$ (cf. \cite[Theorem 8.2]{Soudry1993}). The proof uses the $P_{n+1}$-theory 
developed by Gel'fand-Kazhdan (cf. \cite{GK1972}) and Bernstein-Zelevinsky (\cite{BZ1976}, \cite{BZ1977}), where 
$P_{n+1}$ is the mirobolic subgroup of $\GL_{n+1}(E)$. More precisely, let $P_{n,n}^1\subset P_{n,n}$ be the subgroup 
given by 
\[
P^1_{n,n}=\widehat{\GL}_n(E)\ltimes N_{n,n}.
\]
Then the exact sequence
\[
1\to N_{n,n}\to P^1_{n,n}\to P_{n+1}\to 1
\]
implies that the normalized Jacquet module $(\pi_{N_{n,n}}, \cV_{\pi,N_{n,n}})$ of $\pi$ with respect to $N_{n,n}$ can
be regarded as a representation of $P_{n+1}$. Now the assumption that $\pi$ is of Whittaker type and the proof of 
\cite[Proposition 8.2]{GPSR1987} give the following finite sequence of $P_{n+1}$-modules
\begin{equation}\label{E:P_n+1}
0\subset \cV_0\subset \cV_1\subset\cdots\subset\cV_M=\cV_{\pi,N_{n,n}}
\end{equation}
such that $\cV_{i-1}\backslash\cV_i$ is an irreducible $P_{n+1}$-module for $1\le i\le M$, 
$\cV_1\cong{\rm ind}_{Z_{n+1}}^{P_{n+1}}\psi_{Z_{n+1}}$ and $\cV_1\backslash\cV_M$ does not contain 
${\rm ind}_{Z_{n+1}}^{P_{n+1}}\psi_{Z_{n+1}}$ as a subquotient. Here "ind" stands for the compact induction.
Already at this stage, we can follow the proof of Soudry word for word; however, we shall provide more details in the 
followings.\\

First note that 
\[
{\rm Hom}_{H_r\ltimes Y}(\pi|_{H_r\ltimes Y}\ot\rho_{\tau,s},\psi_Y)
=
{\rm Hom}_{H_r}(\pi_{Y,\psi_Y}\ot\rho_{\tau,s},\mathbb{C})
\]
where $\pi_{Y,\psi_Y}$ is the twisted Jacquet module of $\pi$ with respect to $(Y,\psi_Y)$. Then by the Frobenius 
reciprocity,   
\[
{\rm Hom}_{H_r}(\pi_{Y,\psi_Y}\ot\rho_{\tau,s},\mathbb{C})
\cong
{\rm Hom}_{\GL_r}((\pi_{Y,\psi_Y})_{Y_r}\ot\tau_s,\mathbb{C}).
\]
Observe that the normalized Jacquet module $(\pi_{Y,\psi_Y})_{Y_r}$ of $\pi_{Y,\psi_Y}$ with respect to $Y_r$ is 
isomorphic to $\pi_{\cY,\psi_{\cY}}$, where $\cY:=Y\cdot Y_r$ is a subgroup of $P^1_{n,n}$ (again, the embedding 
in \S\ref{SSS:embed} is used) and we extend $\psi_Y$ trivially across $Y_r$ to obtain a character $\psi_{\cY}$ of $\cY$.  
Moreover, $N_{n,n}$ is normal in $\cY$ and $N_{n,n}\backslash\cY$ is isomorphic to the following subgroup of 
$P_{n+1}$
\[
\cR
=
\stt{
e(x,z,y)
=
\begin{pmatrix}
I_r&0&0\\
x&z&y\\
0&0&1
\end{pmatrix}
\mid
x\in{\rm Mat}_{(n-r)\x r}(E), z\in Z_{n-r}, y\in E^{n-r}
}
\]
and $\psi_\cY$ on $N_{n,n}\backslash\cY$ is mapped to $\psi_{\cR}(e(x,z,y))=\psi_{Z_{n-r}}(z)\psi(y_{n-r})$, where
${}^ty=(y_1,\ldots, y_{n-r})\in E^{n-r}$. The group $\GL_r(E)$ is embedded in $P_{n+1}$ via the standard way, i.e.
\[
a
\mapsto
\pMX{a}{}{}{I_{n-r+1}}\in P_{n+1}
\]
for $a\in\GL_r(E)$. It follows that $\pi_{\cY,\psi_\cY}\cong(\pi_{N_{n,n}})_{\cR,\psi_\cR}$ as $\GL_r(E)$-modules.
By applying Frobenius reciprocity once more, we find that
\begin{align}\label{E:hom isom}
\begin{split}
{\rm Hom}_{\GL_r}((\pi_{Y,\psi_Y})_{Y_r}\ot\tau_s,\mathbb{C})
&\cong
{\rm Hom}_{\GL_r}((\pi_{N_{n,n}})_{\cR,\psi_\cR}\ot\tau_s,\mathbb{C})\\
&\cong
{\rm Hom}_{P_{n+1}}\left(\pi_{N_{n,n}}\ot\left(\delta_{P_{n+1}}\ot{\rm ind}^{P_{n+1}}_{\GL_r(E)\ltimes\cR}
\tau_s\boxtimes\psi_\cR^{-1}\right), \mathbb{C}\right)\\
&\cong
{\rm Hom}_{P_{n+1}}\left(\pi_{N_{n,n}}\ot\left({\rm ind}^{P_{n+1}}_{\GL_r(E)\ltimes\cR}
\tau_{s+1}\boxtimes\psi_\cR^{-1}\right), \mathbb{C}\right)
\end{split}
\end{align}
where in the last line above, we use the fact that the map $f\mapsto\delta_{P_{n+1}}f$ defines an isomorphism 
\[
\delta_{P_{n+1}}\ot{\rm ind}^{P_{n+1}}_{\GL_r(E)\ltimes\cR}
(\tau_s\boxtimes\psi_\cR^{-1})
\cong
{\rm ind}_{\GL_r(E)\ltimes\cR}^{P_{n+1}}(\tau_{s+1}\boxtimes\psi_{\cR}^{-1})
\]
between $P_{n+1}$-modules.\\ 

To proceed, recall that irreducible representations of $P_{n+1}$ are of the form 
\[
{\rm ind}_{\cR_m}^{P_{n+1}}\sigma\boxtimes\psi_{Z_{n+1-m}}
\]
where $\cR_m\subset P_{n+1}$ is the subgroup given by
\[
\cR_m
=
\stt{
\pMX{a}{b}{}{z}\mid a\in\GL_m(E), b\in{\rm Mat}_{m\x(n+1-m)}(E), z\in Z_{n+1-m}
}
\]
for $0\le m\le n$, $\sigma$ is an irreducible representation of $\GL_r(E)$ and we extend the representation 
$\sigma\boxtimes\psi_{Z_{n+1-m}}$ of $\GL_m(E)\x Z_{n+1-m}$ to $\cR_m$ trivially across $b$. 
Then by combining \eqref{E:hom isom} with \eqref{E:P_n+1}, it suffices to show that 
\[
{\rm Hom}_{P_{n+1}}\left(\left({\rm ind}_{\cR_m}^{P_{n+1}}\sigma\boxtimes\psi_{Z_{n+1-m}}\right)
\ot
\left({\rm ind}_{\GL_r(E)\ltimes\cR}^{P_{n+1}}\tau_\xi\boxtimes\psi_{\cR}^{-1}\right),\mathbb{C}
\right)
=
0
\]
for all but finitely many values of $q^{-\xi}$, unless $m=0$, in which case we are dealing with
\[
{\rm Hom}_{P_{n+1}}\left(\left({\rm ind}_{Z_{n+1}}^{P_{n+1}}\psi_{Z_{n+1}}\right)
\ot
\left({\rm ind}_{\GL_r(E)\ltimes\cR}^{P_{n+1}}\tau_\xi\boxtimes\psi_{\cR}^{-1}\right),\mathbb{C}
\right)
\]
which we show to be one-dimensional. Now we are in exactly the same situation as in \cite[Pages 55-57]{Soudry1993}. 
This completes the proof.
\end{proof}

\end{document}